%% file: main.tex
\documentclass[11pt,hyp,]{nyjm}
\usepackage{hyperref}
\hypersetup{nesting=true,debug=true,naturalnames=true}
\usepackage{graphicx,amssymb,upref}

\let\<\langle
\let\>\rangle
\usepackage[all,pdf]{xy}
\UseComputerModernTips

\let\uml\"

\title{Barrier methods for Minimal Submanifolds in the Gibbons-Hawking Ansatz} 

\author{Federico Trinca}  
\address{Mathematical Institute, University of Oxford, Woodstock Road, Oxford, OX2 6GG,
United Kingdom} 
\email{Federico.Trinca@maths.ox.ac.uk}



\usepackage{tikz} \usetikzlibrary{arrows,shapes,patterns} 
\usepackage{systeme}
\usepackage{subcaption}
\usepackage{mathabx}

\newtheorem{thm}{Theorem}[section]
\newtheorem{corollary}{Corollary}[section]
\newtheorem{proposition}{Proposition}[section]
\newtheorem{lemma}{Lemma}[section]

\theoremstyle{definition}
\newtheorem{example}{Example}[section]
\newtheorem{definition}{Definition}[section]
\newtheorem{remark}{Remark}[section]

\newcommand{\R}{\mathbb{R}}
\newcommand{\C}{\mathbb{C}}
\newcommand{\Z}{\mathbb{Z}}
\newcommand{\Tr}{\mathrm{Tr}}

\newcommand{\la}{\langle}
\newcommand{\ra}{\rangle}
\newcommand{\av}[1]{\lvert #1 \rvert}
\newcommand{\mR}{\mathcal{R}}
\newcommand{\mA}{\mathcal{A}}
\newcommand{\mO}{\mathcal{O}}
\DeclareMathOperator{\Length}{Length}
\DeclareMathOperator{\Vol}{Vol}
\newcommand{\Ric}{\mathrm{Ric}}

\renewcommand{\H}{\mathrm{H}}

\begin{document} 
 
\begin{abstract}  
      We describe a barrier argument for compact minimal submanifolds in the multi-Eguchi--Hanson and in the multi-Taub--NUT spaces, which are hyperk\"{a}hler 4-manifolds given by the Gibbons--Hawking ansatz. This approach is used to obtain results towards a classification of compact minimal submanifolds in this setting.  
We also prove a converse of Tsai and Wang's result that relates the strong stability condition to the convexity of the distance function. 
\end{abstract} 
\maketitle
\tableofcontents

\input{content.tex}

            			\bibliographystyle{plain}

\end{document}

%% file: content.tex
\section{Introduction}
In a Riemannian manifold, we say that a submanifold\footnote{Submanifolds and integral varifolds will always be considered without boundary.} is minimal if it is a critical point of the volume functional. As minimal submanifolds are not only of great geometric interest per se, but also encode information on the ambient manifold, these objects are widely studied.

A way to probe compact minimal submanifolds is by using ambient $k$-convex functions. A function $f$ is (strictly) $k$-convex if the sum of the smallest $k$ eigenvalues of $\mathrm{Hess} f$ is everywhere non-negative (positive). Such a function, when restricted to a compact minimal $k$-submanifold $\Sigma$, is subharmonic and hence forces $\Sigma$ to be contained in the set where $f$ is not strict. Given a smooth open domain $\Omega$, we say that $\partial \Omega$ is $k$-convex if the sum of the smallest $k$ eigenvalues of the second fundamental form, pointing inward, is everywhere positive. In this setting, Harvey and Lawson \cite[Theorem 5.7]{HaLa12} constructed a $k$-convex function in the domain, which is strict near $\partial\Omega$. This implies that compact minimal $k$-submanifolds contained in $\Omega$ cannot be tangent to $\partial \Omega$. Hence, $\partial\Omega$ provides a barrier for compact minimal $k$-submanifolds. The parallel with the generalized avoidance principle for the mean curvature flow, which is the gradient flow for the volume functional, is clear \cite[Theorem 14.1]{Wh15}. Moreover, this allow us to extend our results on minimal submanifolds to integral varifolds\footnote{The reader not familiar with the notion of (stationary) integral varifold can read (minimal) "singular" submanifold instead.}.

A hyperk\"{a}hler 4-manifold is a Riemannian manifold $(X,g)$ that is equipped with an $S^2$ of k\"{a}hler structures. This forces the holonomy group of $X$ to be contained in $\mathrm{Sp(1)}\cong\mathrm{SU(2)}$. Hence, hyperk\"{a}hler 4-manifolds are also Calabi--Yau, and so Ricci-flat. Since complex submanifolds of K\"{a}hler manifolds are homologically area minimizing by Wirtinger's inequality, hyperk\"{a}hler 4-manifolds have a distinguished class of minimal submanifolds, namely the complex curves with respect to one of the compatible complex structures. It is easy to see that these complex curves are also special Lagrangians for a Calabi--Yau structure on $X$. Special Lagrangians are not only of great geometric interest, but they also play a crucial role in theoretical physics, particularly in Mirror Symmetry. 

The Gibbons--Hawking ansatz, first introduced in \cite{GH78}, provides a way to construct a family of simply connected hyperk\"{a}hler 4-manifolds with a tri-Hamiltonian circle action. In this family, we have, for example, the Euclidean $\mathbb{R}^4$, the Eguchi--Hanson space, and the Taub--NUT space. As a generalization of these, the Gibbons--Hawking ansatz also gives infinitely many ALE and ALF spaces called multi-Eguchi--Hanson and multi-Taub--NUT respectively, which are characterized by a distribution of points in $\mathbb{R}^3$. Indeed, these are the total space of an $\mathrm{U(1)}$-bundle over $\mathbb{R}^3$ minus finitely many points $\{p_i\}_{i=1}^k$. We denote by $U$ this punctured $\mathbb{R}^3$, parametrized by $\{x_i\}_{i=1}^3$, and by $\pi$ be the projection map of this bundle. The Euclidean and the Taub--NUT space correspond, respectively, to the one-point multi-Eguchi--Hanson and to the one-point multi-Taub--NUT space. The Eguchi--Hanson space correspond to the two-point multi-Eguchi--Hanson case \cite{GH78,Pr79}. 

In this paper, we study the $k$-convexity of natural sets and functions on the multi-Eguchi--Hanson and multi-Taub--NUT spaces, which are all the complete simply connected hyperk\"{a}hler 4-manifolds with a tri-Hamiltonian circle action and finite topology \cite{Bie99}. The barriers that we obtain are used towards a classification of compact minimal submanifolds. Moreover, we show that, apart from the one and the two point case, the natural competitors do not provide, not even locally, a complete description of such objects.

\subsection*{Main results} In the setting above, Lotay and Oliveira \cite{LO20} studied minimal submanifolds that are invariant under the circle action. In particular, they proved the existence of circle-invariant closed geodesics, and that circle-invariant compact minimal surfaces correspond to straight lines connecting two of the characterizing points in $U$. These are also all the compact complex submanifolds. 

It is natural to ask whether all compact minimal submanifolds are circle-invariant, or are contained in one. Indeed, it is well-known that this vacuously holds in the Euclidean $\mathbb{R}^4$ and the Taub--NUT space. A way to prove it is by noticing that circle-invariant spheres around the singular point of $\phi$ are convex with respect to its interior \cite[Appendix B]{LO20}. Moreover, Tsai and Wang \cite[Theorem 5.2]{TW18} proved that the claim is also true in the Eguchi--Hanson case. We use, as barriers, all the circle-invariant ellipsoids of foci the singular points of $\phi$ to extend Tsai and Wang result in the two-point multi-Taub--NUT case.

\begin{thm}\label{IntroTheoremTwoPoints}
Let $(X,g)$ be a multi-Eguchi--Hanson or a multi-Taub--NUT space with two singular points of $\phi$. Then, compact minimal submanifolds are $S^1$-invariant or are contained in the unique $S^1$-invariant compact minimal surface.
\end{thm}

In particular, we proved that, in the multi-Eguchi--Hanson and multi-Taub--NUT spaces with \emph{at most} 2 singular points of $\phi$, compact minimal submanifolds are circle-invariant, or are contained in one.
When we consider at least 3 singular points of $\phi$, we observe that the natural generalization of the sets used above, i.e. ellipsoids with multiple foci, cannot work. 
Instead, we show that circle-invariant spheres and circle-invariant cylinders are $3$-convex for big enough radii. Moreover, for a weaker constant, spheres are also $1$-convex. Unfortunately, this is not true in the cylindrical case. We deduce that compact minimal submanifolds must lie in a certain compact domain containing the characterizing points of the ambient manifold. In the collinear case, this is enough to show the non-existence of compact minimal hypersurfaces. More precisely, we have:

\begin{thm} \label{IntroHypersurfaceSphere}
Let $(X,g)$ be a multi-Eguchi--Hanson or a multi-Taub--NUT space. Compact minimal hypersurfaces need to be contained in $\pi^{-1} (\{x\in U: \av{x}_{\mathbb{R}^3}\leq4/3 \max_i \av{p_i}_{\mathbb{R}^3}\})$. Moreover, there are no compact minimal hypersurfaces contained in $\pi^{-1} (\{x\in U: \av{x}_{\mathbb{R}^3}< \min \{\av{p_i}_{\mathbb{R}^3}: \av{p_i}_{\mathbb{R}^3}>0\}\}).$
\end{thm}

\begin{thm} \label{IntroHyperusrfaceCylinders}
Let $(X,g)$ be a multi-Eguchi--Hanson or a multi-Taub--NUT space. Compact minimal hypersurfaces need to be contained in $\pi^{-1} (\{x\in U: \sqrt{x_1^2+x_2^2}\leq 2 \max_i r_i\})$, where $r^2_i=(p_i)_1^2+(p_i)_2^2$. Moreover, there are no compact minimal hypersurfaces contained in $\pi^{-1} (\{x\in U: \sqrt{x_1^2+x_2^2}< \min \{r_i: r_i>0\}\})$.
\end{thm}

\begin{corollary} \label{IntroCorollaryCollinearCase}
Let $(X,g)$ be a multi-Eguchi--Hanson or a multi-Taub--NUT space with the $\{p_i\}_{i=1}^k$ lying on a line. Then, there are no compact minimal hypersurfaces in $X$.
\end{corollary}

\begin{thm}  \label{IntroHigherCodimensionSpheres}
Let $(X,g)$ be a multi-Eguchi--Hanson or a multi-Taub--NUT space. Compact minimal submanifolds need to be contained $\pi^{-1} (\{x\in U: \av{x}\leq C \max_i\av{p_i}_{\mathbb{R}^3}\})$, where $C\approx 5.07$ is the only real root of the polynomial: $-x^3+4x^2+5x+2$. Moreover, if $p_i=0$ for some $i$, then, there are no compact minimal submanifolds contained in $\pi^{-1} \left(\left\{x\in U: \av{x}< r_0\right\}\right)$, for some $r_0$ small enough. 
\end{thm}

The results discussed so far can be extended to multi-centred Gibbons--Hawking spaces, which are incomplete generalizations of the multi-Eguchi--Hanson and of the multi-Taub--NUT spaces. 

For a generic multi-Eguchi--Hanson or multi-Taub--NUT space we have considered several natural barriers for compact minimal submanifolds. However, these are not enough to prove a result as strong as in the one or two points case. Hence, one would like to find, at least, local barriers around the circle-invariant ones. To this scope, we recall that, in a general Riemannian manifold, the square of the distance function from any strongly stable orientable compact minimal submanifold of dimension $k$ is locally a $k$-convex function \cite[Proposition 4.1]{TW}. Here, a minimal submanifold is said to be strongly stable if the part not involving the Laplacian of the Jacobi operator, $-\mathcal{R}-\mathcal{A}$, is pointwise positive. Strong stability actually characterize the convexity of the square of the distance function. Indeed, we prove the following converse.

\begin{proposition} \label{Final converse Strong Stability}
Let $(X,g)$ be a Riemannian manifold, let $\Sigma\subset M$ be an orientable compact minimal submanifold of dimension $k$ such that $-\mR-\mA$ is a negative operator at a point $p\in \Sigma$, and let $f\in C^{\infty}(\mathbb{R};\mathbb{R})$ increasing. Denoting by $\psi$ the square of the distance function from $\Sigma$, then, for every neighbourhood of $\Sigma$ there exists a point in it where $f\circ\psi$ is not $k$-convex. Moreover, the same holds for every suitable $C^2$-small perturbation of $f\circ \psi$.
\end{proposition}

As in all examples where this method is used \cite{TW, TW1, TW18} the barriers are solely depending on the distance function, we showed that the strong stability condition is equivalent to the existence of natural local barriers.

Going back to the multi-Eguchi--Hanson and multi-Taub--NUT spaces, we observe that strongly stable compact minimal submanifolds need to be $2$-dimensional and also circle-invariant under suitable topological conditions. In particular, we can only consider the circle-invariant surfaces connecting two singular points of $\phi$. If these singular points of $\phi$ are sufficiently separated from the others, then, we prove that the related surface is strongly stable. This is a slight generalization of \cite[Proposition A.1]{LO20}, where we do not assume collinearity.

\begin{proposition} \label{strong stability GH}
Let $(X,g)$ be a multi-Eguchi--Hanson or a multi-Taub--NUT space with $k\geq2$ singular points of $\phi$ $\{p_i\}_{i=1}^k$, let $N$ be a compact $S^1$-invariant minimal surface in $(X,g)$, let $\gamma:=\pi(N)$ be the associated straight line in $U$ connecting $p_1$ and $p_2$, let $q$ be the midpoint of $\gamma$ and let $2a:=\Length_{\mathbb{R}^3}(\gamma)$. Suppose that, for all $i>2$, the Euclidean distance from $q$ to $p_i$ is strictly greater than $(s+1)a$ for $s\geq\max\{\sqrt{{(k-2)}/{2}},R_k\}$, where $R_k$ is the only real root of $-4x^3+16x^2+2x+(k-2)$. Then, $N$ is strongly stable.
\end{proposition}
 
It is easy to see that Proposition \ref{strong stability GH} cannot provide strong stability for all circle-invariant compact minimal surfaces when we have at least 3 singular points of $\phi$.

Finally, we provide a family of multi-Eguchi--Hanson and multi-Taub--NUT spaces with a circle-invariant minimal surface admitting a point where $-\mR-\mA$ is a negative operator. 

\begin{proposition} \label{CounterexampleSS}
Let $(X,g)$ be the multi-Eguchi--Hanson or a multi-Taub--NUT space with singular points of $\phi$ $p_1=(0,0,a)$, $p_2=(0,0,-a)$ and $p_3=(0,\epsilon,0)$, for some $a,\epsilon>0$. Then, fixed $a$ ($\epsilon$) there exists an $\epsilon$ small enough (an $a$ big enough) such that $-\mR-\mA$ is a negative operator at $\pi^{-1}(0)$.
\end{proposition} 

Hence, we have shown that the natural barriers are not strong enough, not even locally, to prove that compact minimal submanifolds are circle-invariant or contained in one for a generic multi-Eguchi--Hanson or multi-Taub--NUT space.

\subsection*{Acknowledgements}
The author wishes to thank his supervisor Jason D. Lotay for suggesting this project and for his enormous help in its development. The author would also like to thank Gon\c{c}alo Oliveira for pointing out a mistake in the first version of this paper and the referee for the useful comments. This work was supported by the Oxford-Thatcher Graduate Scholarship.

\section{The Gibbons--Hawking Ansatz}
In this section, we will describe the Gibbons--Hawking ansatz. We refer to \cite{LO20} and \cite{GW00} for further details. Note that our construction differs by an orientation choice to the one in \cite{LO20}. 

\subsection{Construction} Let $U$ be an open subset of $\mathbb{R}^3$ and let $\pi:{X}\to U$ be a principal $S^1$-bundle over $U$. Let $\xi$ be the infinitesimal generator of the $S^1$ action and let $\eta\in\Omega^1 ({X},\mathbb{R})$ be a connection for the principal bundle, i.e. $\eta$ is $S^{1}$-invariant and satisfies $\eta(\xi)=1$. It is an immediate consequence of these properties, together with Cartan's formula, that $d\eta$ is horizontal and hence $d\eta=\pi^*\alpha$, for some 2-form $\alpha$ on $U$. Let $\phi$ be a positive $\mathbb{R}$-valued function on $U$ satisfying the monopole equation:
\[
\ast_{\mathbb{R}^3} d\phi=\alpha.
\]
Note that, since $d\alpha=0$, the monopole equation forces $\phi$ to be harmonic with respect to the flat metric on $\mathbb{R}^3$. We now construct a hyperk\"{a}hler structure on ${X}$. If $\{x_i\}_{i=1}^3$ are coordinates on $U\subset\mathbb{R}^3$, then we can define: 
\[
\omega_1=dx_1\wedge\eta+\phi dx_2\wedge dx_3,\hspace{5pt}\omega_2=dx_2\wedge\eta+\phi dx_3\wedge dx_1,\hspace{5pt}\omega_3=dx_3\wedge\eta+\phi dx_1\wedge dx_2.
\]
It is straightforward that $\omega_i^2$ are nowhere vanishing and that $\omega_i\wedge\omega_j=0$, for $i\neq j$. These forms are closed, indeed, for $(i, j, k)$ cyclic permutation, the monopole equation implies:
\[
d\omega_i=-dx_i\wedge d\eta+d\phi\wedge dx_j\wedge dx_k=0.
\]
It is clear that these forms, together with the Riemannian metric:
\[
g=\phi^{-1}\eta^2+\phi g_{\mathbb{R}^3},
\]
induce a hyperk\"{a}hler structure on ${X}$.

As in \cite{LO20}, we compute the structure equations. 
\begin{lemma} \label{Structure}
Let $({X},g)$ be a space constructed by the Gibbons--Hawking ansatz using the harmonic function $\phi$. Let $\{e^i\}_{i=0}^3$ be the orthonormal coframe given by: 
\[
e^0=\phi^{-1/2}\eta, \hspace{15pt} e^i=\phi^{1/2}dx_i \hspace{5pt}i=1,...,3.
\]
Then:

\begin{align*}
  &\nabla_{e_0} e_0=\frac{1}{2\phi^{3/2}}\sum_{i=1}^3 \frac{\partial\phi}{\partial x_i}e_i;\\
    &\nabla_{e_i} e_0=-\frac{1}{2\phi^{3/2}}\sum_{j,k=1}^3 \epsilon_{ijk} \frac{\partial\phi}{\partial x_j}e_k;\\
    \end{align*}
    \begin{align*}
    &\nabla_{e_0} e_i=-\frac{1}{2\phi^{3/2}}\left( \frac{\partial\phi}{\partial x_i}e_0+\sum_{j,k=1}^3 \epsilon_{ijk} \frac{\partial\phi}{\partial x_j}e_k\right);\\
    &\nabla_{e_i} e_j=\frac{1}{2\phi^{3/2}}\left( \frac{\partial\phi}{\partial x_j}e_i-\sum_{k=1}^3 \left(\epsilon_{ijk} \frac{\partial\phi}{\partial x_k}e_0+\delta_{ij}\frac{\partial\phi}{\partial x_k} e_k\right)\right),
\end{align*}

where $\epsilon_{ijk}$ is the permutation symbol and $\{e_i\}_{i=1}^3$ is the orthonormal frame dual to $\{e^i\}_{i=1}^3$.
\end{lemma}
\begin{proof}
For a proof, see \cite[Lemma 2.2]{LO20}. Note that the differences arise from the different choice of the sign of $\eta$.
\end{proof}
\subsection{Examples} Here we describe the spaces needed in the following sections.

\begin{example} [Flat metric] Let $U=\mathbb{R}^3\setminus \{0\}$ and let $\phi=1/2r$, where $r=\av{x}_{\mathbb{R}^3}$. By the substitution $\rho=\sqrt{2r}$, we can see that $({X},g)$ is the description in polar coordinates of $(\mathbb{R}^4\setminus \{0\},g_{\mathbb{R}^4})$. It is clear that we can extend the metric $g$ to 0 and obtain the whole $\mathbb{R}^4$. 
\end{example}

\begin{example} [Eguchi--Hanson metric]
Let $p_1, p_2$ be two points in $\mathbb{R}^3$ and let $U=\mathbb{R}^3\setminus\{p_1,p_2\}$. If we define $\phi$ as follows:
\[
\phi=\frac{1}{2\av{x-p_1}_{\mathbb{R}^3}}+\frac{1}{2\av{x-p_2}_{\mathbb{R}^3}},
\]
we obtain the Eguchi--Hanson metric. Once again, we can add back $p_1$ and $p_2$. Usually, the Eguchi--Hanson metric is described as a metric on $T^*S ^2$. An explicit isometry can be found in \cite{Pr79}.
\end{example}

\begin{example} [Multi-Eguchi--Hanson metric] 
Let $\{p_i\}_{i=1}^k$ be $k$ points in $\mathbb{R}^3$ and let $U=\mathbb{R}^3\setminus\{p_i\}_{i=1}^k$. If we define $\phi$ as follows:
\[
\phi=\sum_{i=1}^k \frac{1}{2\av{x-p_i}_{\mathbb{R}^3}},
\]
we obtain the multi-Eguchi--Hanson metric. Analogously to the Eguchi--Hanson metric, we can add back the points removed. 
\end{example}

\begin{example} [Taub--NUT metric]
Let $m$ be a positive real number and let $U=\mathbb{R}^3\setminus\{0\}$. If we define $\phi$ as follows:
\[
\phi=m+\frac{1}{2\av{x}_{\mathbb{R}^3}},
\]
we obtain the Taub--NUT metric. We can add back $0$ and obtain topologically $\mathbb{R}^4$. 
\end{example}

\begin{example} [Multi-Taub--NUT metric]
Let $m$ be a positive real number, let $\{p_i\}_{i=1}^k$ be $k$ points in $\mathbb{R}^3$ and let $U=\mathbb{R}^3\setminus\{p_i\}_{i=1}^k$. If we define $\phi$ as follows:
\[
\phi=m+\sum_{i=1}^k \frac{1}{2\av{x-p_i}_{\mathbb{R}^3}},
\]
we obtain the multi-Taub--NUT metric. As above, we can add back the points removed.
\end{example}

\begin{example} [Multi-centred Gibbons--Hawking space]
Let $m$ be a non-negative real number, let $\{p_i\}_{i=1}^k$ be $k$ points in $\mathbb{R}^3$, let $\{c_i\}_{i=1}^k\subset\mathbb{N}$ and let $U=\mathbb{R}^3\setminus\{p_i\}_{i=1}^k$. If we define $\phi$ as follows:
\[
\phi=m+\sum_{i=1}^k \frac{c_i}{2 \av{x-p_i}_{\mathbb{R}^3}},
\]
we obtain the multi-centred Gibbons-Hawking space. Unless $c_i=1$, it is not possible to add back the points removed.
\end{example}

\section{Minimal Submanifolds} \label{SectionBarrier}
Let $(M,\la\cdot,\cdot\ra)$ be a Riemannian manifold. 
\begin{definition}
A k-dimensional submanifold $\Sigma$ of $M$ is minimal if it is a critical point of the volume functional. By the first variation formula \cite[Theorem 2.4.1]{Sim68}, $\Sigma$ is minimal if and only if  $\H:=\sum_{i=1}^k A(e_i,e_i)$, the mean curvature of $\Sigma$, vanishes. $A$ is the map defined by $A(X,Y):=\nabla_X^\perp Y$ for $X, Y$ vector fields tangent to $\Sigma$, and $\{e_i\}_{i=1}^k$ is a local orthonormal frame of $\Sigma$.
\end{definition}

\subsection{Barriers for minimal submanifolds} \label{Subsectionbarrier}
 
\begin{definition} \label{def k-convex}
A function $f:M\to\mathbb{R}$ is said to be $k$-convex (or $k$-plurisubharmonic) if 
$$\Tr_W \mathrm{Hess} f_x\geq0 \hspace{10pt} \forall x\in M, \hspace{2pt} \forall\hspace{1pt} W\in G(k,T_x M),$$
where $G(k,T_x M)$ is the Grassmannian of $k$-dimensional subspaces of $T_x M$. If the inequality is strict in a set we will say that $f$ is strictly $k$-convex there.
\end{definition}

The following well-known lemma shows that a $k$-convex function is subharmonic when restricted to a $k$-dimensional minimal submanifold. 
\begin{lemma} \label{kconvrestr}
Let $f:M\to\mathbb{R}$ be a $k$-convex function. Then, any orientable $k$-dimensional compact minimal submanifold $\Sigma$ of $M$ is contained in the set where $f$ is not strict. In particular, $f$ is constant on every connected component of $\Sigma$.
\end{lemma}
\begin{proof}
Let $\Sigma$ be an orientable $k$-dimensional compact minimal submanifold of $M$. We immediately have that: 
\[
\Tr_{\Sigma} \mathrm{Hess} f=\Delta_{\Sigma}f-\H(f),
\]
where $\H$ is the mean curvature vector of $\Sigma$, and $\Delta_{\Sigma}$ is the Laplace operator of the induced metric on $\Sigma$. It follows from minimality and $k$-convexity that $\Delta_\Sigma f\geq0$. The maximum principle gives the lemma.
\end{proof}

Let $\Omega\subset M$ be a domain with smooth non-empty boundary $\partial \Omega$.

\begin{definition} \label{def boundary convex}
The boundary $\partial\Omega$ is said to be $k$-convex if 
$$\Tr_W  \Gemini_x\geq0 \hspace{10pt} \forall x\in \partial\Omega, \hspace{2pt} \forall\hspace{1pt} W\in G(k,T_x \partial\Omega),$$
where $\Gemini$ is the second fundamental form of the hypersurface $\partial \Omega$ with respect to the inward pointing normal $\nu$, i.e. $\Gemini(X,Y):=\la A(X,Y), \nu\ra$ for all $X,Y$ vectors tangent to $\partial \Omega$. If the inequality is strict in a set we will say that $\partial\Omega$ is strictly $k$-convex there.
\end{definition}

\begin{remark} \label{LevelSetsConvexFunctions}
Let $M$ be an orientable manifold, let $f$ be a (strictly) $k$-convex function of $M$ and let $a$ be a regular value of $f$. The well-known formula for the second fundamental form of the hypersurface $f^{-1}(a)$:
\[
\Gemini=\frac{1}{\av{\nabla f}} \mathrm{Hess} f,
\]
implies that $f^{-1}(a)$ is a (strictly) $k$-convex hypersurface. \end{remark}

Harvey and Lawson obtained a sort of converse of this remark. 
\begin{thm}[Harvey and Lawson {\cite[Theorem 5.7]{HaLa12}}] \label{boundarytointerior}
Let $\partial\Omega$ be everywhere strictly $k$-convex. Then, there is a k-convex function $f\in C^{\infty}(\overline{\Omega})$ that is strict in a neighbourhood of $\partial\Omega$. This function can be constructed such that it constantly achieves its maximum at $\partial\Omega$. \end{thm}

\begin{remark}
Theorem \ref{boundarytointerior} is going to be crucial in our discussion. 
Indeed, it will allow us to reduce the problem of $1$ dimension. 
\end{remark}

\begin{corollary} \label{BarrierSmooth}
Let $\partial\Omega$ be strictly $k$-convex. Then, there are no orientable $k$-dimensional compact minimal submanifolds contained in $\Omega$ with a point tangent to $\partial\Omega$.
\end{corollary}
\begin{proof}
It is a straightforward consequence of Theorem \ref{boundarytointerior} and Lemma \ref{kconvrestr}.
\end{proof}

\begin{remark}
When $k=\dim M-1$, $\partial\Omega$ is $k$-convex if and only if it has mean curvature pointing inward. In this setting, Corollary \ref{BarrierSmooth} can be viewed as a direct consequence of the classical avoidance principle for the mean curvature flow. 
\end{remark}

It is well-known that the trace conditions in Definition \ref{def k-convex} and in Definition \ref{def boundary convex} are actually restrictions on the sum of the smallest eigenvalues of the associated matrix. 

\begin{lemma} \label{Linear Algebra translation}
Let $A\in Sym_n(\mathbb{R})$, with ordered eigenvalues $\lambda_1\leq...\leq \lambda_n$. Then, 
$$
\inf_{W\in G(k,\mathbb{R}^n)} \Tr_W A=\lambda_1+...+\lambda_k.
$$

\end{lemma}
\begin{remark}
It is obvious that $k$-convexity implies $l$-convexity for all $l\geq k$. $1$-convexity will be simply called convexity. 
\end{remark}

Similarly to \cite{LS20}, we can use the generalized barrier principle \cite[Theorem 14.1]{Wh15} to extend previous results to the geometric measure theory setting. In this way, we can also drop the orientability condition in Corollary \ref{BarrierSmooth}. We recall that the integral Brakke flow is a weak version of the mean curvature flow, where stationary integral varifolds are constant solutions.

\begin{corollary} \label{BarrierGMT}
Let $\partial\Omega$ be strictly $k$-convex. Then, there are no stationary compactly supported integral varifolds of dimension $k$ contained in $\Omega$ with support intersecting $\partial\Omega$.
\end{corollary}
\begin{proof}
Assume by contradiction that there exists such integral varifold $V$ with support $T$, and let ${u}$ be the function constructed in Theorem \ref{boundarytointerior}. Applying \cite[Theorem 14.1]{Wh15}  to the constantly $V$ Brakke flow and to the function $u$, which is independent from time, we have that $u$ restricted to $T$ cannot have a maximum at the points of $\partial\Omega\cap T\neq\emptyset$. This contradicts Theorem \ref{boundarytointerior}.
\end{proof}

\subsection{Strong stability} 
We now focus on the second variation of the volume.
\begin{definition} \label{definition stability}
A minimal submanifold $\Sigma$ of $M$ is stable if the second variation is a non-negative quadratic form. By the second variation formula \cite[Theorem 3.2.2]{Sim68}, $\Sigma$ is stable if and only if 
\begin{align*}
\int_{\Sigma} \av{\nabla^\perp V}^2 -\la \mR(V), V\ra  - \la\mA(V),V\ra\geq 0,
\end{align*}
for all $V$, compactly supported vector fields normal to $\Sigma$. Here, $\mR$ is the normal trace of the Riemann tensor,  $\mR(V):=\Tr_{\Sigma} (R_M(\cdot ,V)\cdot)^{\perp}$, and $\mA(V)$ is the Simons' operator, which can be expressed, in a local orthonormal frame $\{e_i\}$ of $\Sigma$, as $\mA(V)=\sum_{i,j}\la A(e_i,e_j),V\ra$ $A(e_i,e_j)$.
\end{definition}

We now deal with a stronger condition than stability, which was first studied by Tsai and Wang in \cite{TW}. This condition is strictly related to subsection \ref{Subsectionbarrier}.
\begin{definition}
A minimal submanifold $\Sigma$ of $M$ is said to be strongly stable, if $-\mR-\mA$ is a (pointwise) positive operator on the normal bundle of $\Sigma$.
\end{definition}

\begin{remark}
It is clear that strongly stable submanifolds are in particular stable.
\end{remark}

In hyperk\"{a}hler 4-manifolds, the strong stability condition for surfaces greatly simplifies.
\begin{proposition} [Tsai and Wang {\cite[Appendix A.1]{TW}}] \label{Strong Stability in HK}
Let $(M,g)$ be a 4-dimensional hyperk\"{a}hler manifold and let $\Sigma$ be a minimal surfaces in $M$. Then, $\Sigma$ is strongly stable if and only if the Gaussian curvature of $\Sigma$ is everywhere positive.
\end{proposition}
\begin{proof}

For a proof see \cite[Appendix A.1]{TW} and the special Lagrangian type argument of \cite[Proposition 3.1]{TW1}.
\end{proof}

\begin{corollary} \label{Strongly Stable surfaces are Spheres}
Let $(M,g)$ be a 4-dimensional hyperk\"{a}hler manifold and let $\Sigma$ be a strongly stable orientable compact minimal surface in $M$. Then, $\Sigma$ is topologically a sphere. 
\end{corollary}
\begin{proof}
Proposition \ref{Strong Stability in HK} implies that $\Sigma$ has positive Gaussian curvature. Gauss--Bonnet theorem implies that $\Sigma$ needs to be a sphere. 
\end{proof}

We now highlight the connection between strong stability and barriers. 
\begin{proposition} [Tsai and Wang {\cite[Proposition 4.1]{TW}}] \label{TW local uniqueness}
Let $\Sigma\subset M$ be a strongly stable orientable compact minimal submanifold of dimension $k$. Then, there exists a neighbourhood of $\Sigma$ such that the square of the distance function from $\Sigma$ is $k$-convex in such a neighbourhood. Moreover, it is strict outside $\Sigma$.
\end{proposition}

We now show that a converse holds.

\begin{proposition} \label{converse strong stability}
Let $\Sigma\subset M$ be an orientable compact minimal submanifold of dimension $k$ such that $-\mR-\mA$ is a negative operator at a point $p\in \Sigma$. Denoting by $\psi$ the square of the distance function from $\Sigma$, then, for every neighbourhood of $\Sigma$ there exists a point in it where $\psi$ is not $k$-convex.
\end{proposition}
\begin{proof}
Let $\{e_1,...,e_k,e_{k+1},...,e_n\}$ be the orthonormal "partial" geodesic frame in a neighbourhood of $p$ in $M$ as in \cite[Section 2.2]{TW}. Essentially, this frame is constructed as follows:
\begin{enumerate}
\item Let $\{e_1,...,e_k\}$ be an oriented orthonormal basis of $T_p \Sigma$. By using the parallel transport with respect to $\nabla^T$ along the radial geodesics of $\Sigma$, we obtain a local orthonormal frame of $T\Sigma$ in a neighbourhood of $p$ in $\Sigma$. We still denote this frame by $\{e_1,...,e_k\}$.
\item  Let $\{e_{k+1},...,e_n\}$ be an orthonormal basis of $N_p\Sigma$. By using the parallel transport with respect to $\nabla^\perp$ along radial geodesics of $\Sigma$, we obtain a local orthonormal frame of $N\Sigma$ in a neighbourhood of $p$ in $\Sigma$.  We still denote this frame by $\{e_{k+1},...,e_n\}$.
\item Finally, given the local orthonormal frame for $TM|_\Sigma$ constructed before $\{e_1,...,e_n\}$, we use the parallel transport with respect to $\nabla$ along the normal geodesics to obtain a local orthonormal frame of $TM$ in a neighbourhood of $p$ in $M$. We still denote this frame by $\{e_1,...,e_n\}$.
\end{enumerate}
Let $\{\omega^1,...,\omega^n\}$ be the dual coframe. It is clear that, by using the exponential map in a similar way, we also obtain local "partial" geodesic coordinates, which we denote by $(x_1,...,x_k,y_{k+1},...,y_n)$. Observe that $\psi=\sum_{i=k+1}^n (y_i)^2$ and that $d\psi=2\sum_{i=k+1}^n (y_i) \omega^i$.

At any point $(0,y)$, consider the $k$-plane $L:=\mathrm{span} \{e_1,...,e_k\}$. We claim that, for $\av{y}$ small enough, $\Tr_L \mathrm{Hess} \psi<0$. Since $e_j(\psi)=d\psi(e_j)\equiv 0$ for all $j<k$, we can use \cite[Proposition 2.6]{TW} as in \cite[Proposition 4.1]{TW} to obtain:
\begin{align*}
\Tr_L \mathrm{Hess} \psi&=2\left(\left\la (-\mR-\mA)_p(\sum_{\alpha=k+1}^n y_\alpha e_\alpha),\sum_{\alpha=k+1}^n y_\alpha e_\alpha \right\ra\right)+\mO(\av{y}^3)\\
&\leq -c_0 \av{y}^2+C\av{y}^3, 
\end{align*}
where $c_0$ and $C$ are positive constants. The inequality follows from negativity of $(-\mR-\mA)_p$.

It is clear that, for $\av{y}$ small enough, $\Tr_L \mathrm{Hess} \psi<0$. 
\end{proof}

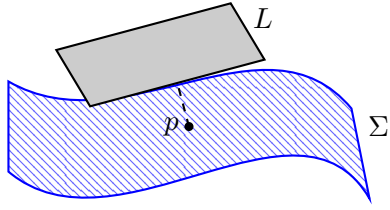
\begin{figure}[t]
\begin{tikzpicture} [scale=1.2]

\draw [-,blue,thick,pattern=north west lines, pattern color=blue!60] (-2,-0.5) to [out=-40,in=120] (2,-0.8) -- 
 (1.8,0.2) to[out=130,in=-30]  (-2,0.5)  -- cycle;
\fill[black] (0,0) circle (0.05cm) node[black,left] {$p$};
\node at (2.1,0) {$\Sigma$};

\draw[thick,dashed, rotate=15] (0,0) -- (0,1);

\fill[gray!40, rotate=15] (-1.2,1.2)-- (-1,0.5)-- (1,0.5) -- (0.8,1.2)-- cycle;
\draw[thick, rotate=15] (-1.2,1.2)-- (-1,0.5)-- (1,0.5) -- (0.8,1.2)-- cycle;

\node at (0.6,1.2) [black, right] {$L$};
\end{tikzpicture}
\caption{Plane that violates convexity in Proposition \ref{converse strong stability}.}\label{Plane Converse TW}
\end{figure}

\begin{proof}[Proof of Proposition \ref{Final converse Strong Stability}]
The chain rule yields: 
\[
\mathrm{Hess} (f\circ\psi)=f'(\psi) \mathrm{Hess} \psi +f''(\psi) \nabla \psi \circ \nabla \psi.
\]
Let $L$ be the plane as in the Proposition \ref{converse strong stability}. Since $\nabla \psi=2\sum_{i=k+1}^n (y_i) \omega^i \equiv 0$ on $L$, we have:
\begin{align*}
\Tr_L \mathrm{Hess} (f\circ\psi)&=f'(\psi) \Tr_L \mathrm{Hess} \psi\leq f'(\psi)(-c_0 \av{y}^2+C\av{y}^3),
\end{align*}
where $c_0$ and $C$ are the same positive constants of the proof of Proposition \ref{converse strong stability}. It follows that $\Tr_L \mathrm{Hess} (f\circ\psi)<0$ for $\av{y}$ small enough. 

Let $\epsilon\geq0$ and let $h\in C^{\infty}$, with $\sum_{l=0}^2 \av{\nabla^l h}_{(0,y)}\leq \epsilon \av{y}^2$. Then:
\begin{align*}
\Tr_L \mathrm{Hess} (f\circ\psi+h)=\Tr_L \mathrm{Hess} (f\circ\psi)+\Tr_L \mathrm{Hess} h<0,
\end{align*}
for $\av{y}$ and $\epsilon$ small enough. 

This implies that  $f\circ\psi+h$ is not $k$-convex and we can conclude.
\end{proof}

\begin{remark}
As we know where the non-convex points will occur, Proposition \ref{Final converse Strong Stability} holds for every function that is close to a function of the distance in a neighbourhood of $p$. 
\end{remark}

\begin{corollary} \label{converse Level Sets}
Let $\Sigma\subset M$ be an orientable compact minimal submanifold of dimension $k$ such that $-\mR-\mA$ is a negative operator at a point $p\in \Sigma$. Denoting by $\psi$ the square of the distance function from $\Sigma$, then, the level sets of $\psi$, corresponding to small enough values of $\psi$, are not $k$-convex with respect to the domain containing $\Sigma$.
\end{corollary}
\begin{proof}
Since an outward normal to any level set of $\psi$ is $\nabla\psi$, it follows that the $k$-plane $L$, as chosen in the proof of Proposition \ref{converse strong stability}, is formed by tangent vectors to the level set. We conclude by Proposition \ref{converse strong stability} and the well-known formula: 
\[
\Gemini=\frac{1}{\av{\nabla \psi}} \mathrm{Hess}\psi,
\]
where $\Gemini$ is the second fundamental form of the level set with respect to the inward pointing normal. 
\end{proof}

\begin{remark}
Given $f$ as in Proposition \ref{Final converse Strong Stability}, the same holds for $f\circ\psi$ instead of $\psi$. It is clear that this does not matter as $f\circ\psi$ and $\psi$ have the same level sets.
\end{remark}

\subsection{Minimal submanifolds and the Gibbons--Hawking ansatz}
We will now discuss minimal submanifolds in the spaces constructed via the Gibbons--Hawking ansatz. We first deal with hypersurfaces.

\begin{lemma} \label{S1invarianthypersurfaces}
Let $(X,g)$ be a space given by the Gibbons--Hawking ansatz associated to the harmonic function $\phi$ on $U\subset\mathbb{R}^3$, let $N$ be an $S^1$-invariant hypersurface in $(X,g)$ and let $\Sigma:=\pi(N)$ be the associated surface in $U$. Then, $$\Vol_X (N)=2\pi \Vol_{(U,\phi^{1/2}g_{\mathbb{R}^3})}(\Sigma).$$ Moreover, $N$ is minimal in $(X,g)$ if and only if $\Sigma$ is minimal in $(U,\phi^{1/2}g_{\mathbb{R}^3})$.
\end{lemma}
\begin{proof}
Let $\alpha,\beta$ be a positively oriented local orthonormal coframe of $\Sigma\subset U$ with respect to the Euclidean metric. It is clear, under the obvious identification, that $\{\phi^{1/2}\alpha,\phi^{1/2}\beta, \phi^{-1/2}\eta\}$ is a positive oriented local orthonormal coframe for $N$. At the level of volume forms we have: 
\[
d\Vol_N=\phi^{1/2}\alpha\wedge\beta\wedge\eta.
\]
Integrating, the desired formula follows easily.  

By \cite[Theorem 1]{HsLa71}, we see that $N$ is minimal if and only if it is stationary with respect to compactly supported $S^1$-equivariant variations. It is clear that compactly supported $S^1$-equivariant variations correspond to compactly supported variations of $U$.
\end{proof}

\begin{example}
To the knowledge of the author, the only known examples of circle-invariant minimal hypersurfaces in the spaces constructed by the Gibbons--Hawking ansatz are the totally geodesic hypersurfaces given by the symmetries of $(U,\phi)$. 
For example, in the multi-Eguchi--Hanson and in the multi-Taub--NUT spaces, if the singular points of $\phi$ admit a plane of symmetry, that plane correspond to a circle-invariant, possibly singular when containing characterizing points, minimal hypersurface.
\end{example}

\begin{remark}
By the formula for the scalar curvature under a conformal change of metric and by the harmonicity of $\phi$, we observe that $(U,\phi^{1/2} g_{\mathbb{R}^3})$ is incomplete and has positive scalar curvature.
\end{remark}

\begin{lemma}
There are no strongly stable minimal hypersurfaces in the spaces constructed via the Gibbons--Hawking ansatz.
\end{lemma}
\begin{proof}
In the hypersurface case, it is well known that the operator simplifies as:
$$(-\mR-\mA)(\nu)=(-\Ric(\nu,\nu)-\av{A}^2),$$
where $\nu$ is a local unit normal. As the spaces constructed by the Gibbons--Hawking ansatz are hyperk\"{a}hler, and hence Ricci-flat, we immediately see that the operator cannot be positive.
\end{proof}

We now turn our attention to surfaces.

\begin{lemma}[Lotay and Oliveira {\cite[Lemma 4.1]{LO20}}] \label{S1invariantsurfaces}
Let $(X,g)$ be a space given by the Gibbons--Hawking ansatz associated to the harmonic function $\phi$ on $U\subset\mathbb{R}^3$, let $N$ be an $S^1$-invariant surface in $(X,g)$ and let $\gamma:=\pi(N)$ be the associated curve in $U$. Then, $$\Vol_X (N)=2\pi \Length_{(U,g_{\mathbb{R}^3})}(\gamma).$$ Moreover, $N$ is minimal in $(X,g)$ if and only if $\gamma$ is a geodesic (i.e. a straight segment) in $(U,g_{\mathbb{R}^3})$.
\end{lemma}
\begin{proof}
It follows as in Lemma \ref{S1invarianthypersurfaces}.
\end{proof}

\begin{example} \label{Level set example}
In the multi-Eguchi--Hanson and in the multi-Taub--NUT spaces, every $S^1$-invariant surface projecting to a straight line is minimal. If such a line connects two of the singular points of $\phi$, it is clear that the related surface is compact and topologically a sphere. In the Eguchi--Hanson case, this segment corresponds to the zero section of $T^\ast S^2$ and the level sets of the distance function from it are ellipsoids in the Gibbons--Hawking ansatz setting \cite{Pr79} (see Figure \ref{Equivalence EH}).
\end{example}

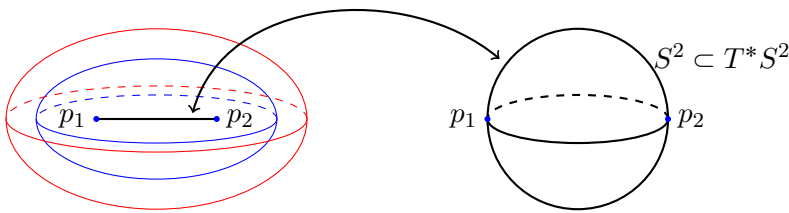
\begin{figure}[t] 
\begin{tikzpicture} [scale=0.8]
\draw[thick] (-1,0,0)--(1,0,0);
\fill[blue] (-1,0,0) circle (0.05cm) node[black,left] {$p_1$};
\fill[blue] (1,0,0) circle (0.05cm) node[black,right] {$p_2$};
\draw[blue] (0,0,0) ellipse (2cm and 1cm);
\draw[blue, dashed] (2,0,0) arc(0:180:2cm and 0.4cm);
\draw[blue] (-2,0,0) arc(180:360:2cm and 0.4cm);
\draw[red] (0,0,0) ellipse (2.5cm and 1.5cm);
\draw[red, dashed] (2.5,0,0) arc(0:180:2.5cm and 0.55cm);
\draw[red] (-2.5,0,0) arc(180:360:2.5cm and 0.55cm);

\draw[thick] (7,0,0) circle (1.5 cm);
\node at (9.4,1.05,0) {$S^2\subset T^\ast S^2$};
\draw[thick] (5.5,0,0) arc (180:360:1.5 cm and 0.4cm);
\draw[thick, dashed] (8.5,0,0) arc (0:180:1.5 cm and 0.4cm);
\fill[blue] (5.5,0,0) circle (0.05cm) node[black,left] {$p_1$};
\fill[blue] (8.5,0,0) circle (0.05cm) node[black,right] {$p_2$};

\draw[thick, <->] (0.6,0.1) to [out=75,in=140] (5.7,1);
\end{tikzpicture}
\caption{Equivalence of Eguchi--Hanson metric to two centre Gibbons--Hawking metric.}\label{Equivalence EH}
\end{figure}

In the Eguchi--Hanson space, Tsai and Wang completely characterized compact minimal submanifolds. Indeed, they showed that they are contained in the only circle-invariant compact minimal surface.  

\begin{lemma}[Tsai and Wang {\cite{TW18}}] \label{EHconvexfunction}
In the Eguchi--Hanson space, the square of the distance function from the unique $S^1$-invariant minimal surface is strictly convex.
\end{lemma}

\begin{thm}[Tsai and Wang {\cite{TW18}} for the smooth case, Lotay and Schulze {\cite{LS20}} for the GMT case] \label{TheoremnonexistenceEH}
Let $(X,g)$ be the Eguchi-Hanson space. Then, compact minimal submanifolds (compactly supported stationary integral varifolds) are contained in the unique $S^1$-invariant compact minimal surface.

\end{thm}

Lotay and Oliveira observed that all $S^1$-invariant minimal surfaces are holomorphic with respect to a compatible complex structure.
\begin{proposition}[Lotay and Oliveira {\cite[Lemma 4.3]{LO20}}] \label{S^1 invariant geodesics are calibrated}
Let $(X,g)$ be a space given by the Gibbons--Hawking ansatz associated to the harmonic function $\phi$ on $U\subset\mathbb{R}^3$, let $N$ be an $S^1$-invariant minimal surface in $(X,g)$ and let $\gamma:=\pi(N)$ be the associated curve in $U$ parametrized by arc-length. Then, $N$ is an holomorphic curve with respect to:
\[
\omega_{\dot{\gamma}}=\sum_{i=1}^3 \dot{\gamma}_i \left(dx_i\wedge \eta+\phi dx_j\wedge dx_k\right),
\]
where $(i,j,k)$ is a cyclic permutation of $(1,2,3)$. In particular, all $S^1$-invariant minimal surfaces are calibrated, i.e. there exists a closed form of the ambient manifold that restricts to the volume form of the surface.
\end{proposition}
\begin{proof}
Note that $\{\phi^{-1/2} \dot{\gamma},\phi^{1/2} \xi\}$ is a local  orthonormal frame of $N$. Plugging it in $\omega_{\dot{\gamma}}$, we get: 
\[
\omega_{\dot{\gamma}} (\phi^{-1/2}\dot{\gamma},\phi^{1/2}\xi)=\sum_{i=1}^3 \dot{\gamma}^2_i=\av{\dot\gamma}_{\mathbb{R}^3}=1.
\]
\end{proof}

In the multi-Eguchi--Hanson or multi-Taub--NUT space, under the compactness hypothesis, the converse holds.
\begin{proposition}\label{Compact Complex curves}
Let $(X,g)$ be a multi-Eguchi--Hanson or multi-Taub--NUT space and let $N$ be a compact holomorphic curve with respect to one of the compatible complex structures. Then, $\pi(N)$ is contained in the union of the lines connecting the singular points of $\phi$. 
\end{proposition}
\begin{proof}
Let $\Sigma$ be a compact holomorphic curve of $X$. As $X$ is the crepant resolution of $\C^2 /\Z_k$, the projection of $\Sigma$ to $\C^2 /\Z_k$ is also compact and holomorphic. Since there are no non-trivial compact holomorphic curves in $\C^2 /\Z_k$, the projection of $\Sigma$ needs to be contained in the preimage of the singular set. 
\end{proof}

\begin{remark}
The compactness assumption is crucial. Indeed, by the bundle construction of calibrated submanifolds in  the Eguchi--Hanson space \cite[Theorem 3.1]{KM05}, we see that holomorphic curves, and in particular minimal surfaces, need not be $S^1$-invariant or with projection contained in a plane of $\mathbb{R}^3$ (see Figure \ref{KM counterexample}). \end{remark}

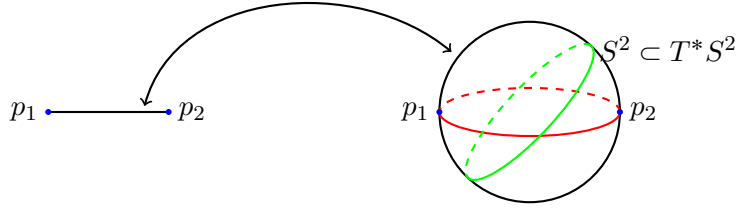
\begin{figure}[t]
\begin{tikzpicture}[scale=0.8]
\draw[thick] (-1,0,0)--(1,0,0);
\fill[blue] (-1,0,0) circle (0.05cm) node[black,left] {$p_1$};
\fill[blue] (1,0,0) circle (0.05cm) node[black,right] {$p_2$};

\draw[thick] (7,0,0) circle (1.5 cm);
\node at (9.3,1.05,0) {$S^2\subset T^\ast S^2$};
\draw[red,thick] (5.5,0,0) arc (180:360:1.5 cm and 0.4cm);
\draw[red,thick, dashed] (8.5,0,0) arc (0:180:1.5 cm and 0.4cm);
\fill[blue] (5.5,0,0) circle (0.05cm) node[black,left] {$p_1$};
\fill[blue] (8.5,0,0) circle (0.05cm) node[black,right] {$p_2$};
\draw[green,thick, rotate around={47:(7,0,0)}] (5.5,0,0) arc (180:360:1.5 cm and 0.4cm);
\draw[green,thick, dashed,rotate around={47:(7,0,0)}] (8.5,0,0) arc (0:180:1.5 cm and 0.4cm);
\draw[thick, <->] (0.6,0.1) to [out=75,in=140] (5.7,1);
\end{tikzpicture}
\caption{Base submanifolds for the bundle construction of holomorphic curves that are not circle invariant (in red) or with projection contained in a plane of $\mathbb{R}^3$ (in green).} \label{KM counterexample}
\end{figure}

\begin{corollary}
	Let $(X,g)$ be a multi-Eguchi--Hanson or multi-Taub--NUT space. Then, every $S^1$-invariant surface connecting two singular points of $\phi$ is the unique volume-minimizer in the homology class.
\end{corollary}
\begin{proof}
	Proposition \ref{S^1 invariant geodesics are calibrated} implies that such a surface is calibrated, and hence volume-minimizing in the homology class. Moreover, every other minimizer needs to be a compact complex submanifold with respect to the same complex structure. We can conclude via Proposition \ref{Compact Complex curves}.
\end{proof}

We now show that, under suitable topological conditions, the only compact embedded orientable strongly stable minimal surfaces of the multi-Eguchi--Hanson and multi-Taub--NUT spaces are circle-invariant. 
\begin{proposition}
Let $(X,g)$ be a multi-Eguchi--Hanson or multi-Taub--NUT space and let $N$ be an embedded stable minimal sphere in the same $H_2(X,\Z)$ homology class of an $l$-chain of $S^1$-invariant minimal spheres, $l\geq0$. Then, $N$ is a complex submanifold with respect to one of the complex structures on $M$ compatible with the metric.
\end{proposition}
\begin{proof}
Let $\nu$ be the normal bundle of $N$. By the embeddedness of $N$ and by the topological condition we have: 
\[
\chi(\nu)=Int([N], [N])\geq-2,
\]
where $\chi(\nu)$ is the Euler number of $\nu$ and $Int$ is the intersection form of $X$.
We conclude by 	\cite[Corollary 5.4]{MW93}.
\end{proof}

\begin{corollary}
Let $(X,g)$ be a multi-Eguchi--Hanson or multi-Taub--NUT space and let $N$ be an orientable embedded compact strongly stable minimal surface in the same $H_2(X,\Z)$ homology class of an $l$-chain of $S^1$-invariant minimal spheres, $l\geq0$. Then, $N$ is a complex submanifold with respect to one of the complex structures on $M$ compatible with the metric.
\end{corollary}
\begin{proof}
It follows from previous proposition and Corollary \ref{Strongly Stable surfaces are Spheres}.
\end{proof}

Finally, we consider geodesics.

\begin{lemma}[Lotay and Oliveira {\cite[Lemma 3.1]{LO20}}] \label{S1invariantgeodesics}
Let $(X,g)$ be a multi-Eguchi--Hanson or multi-Taub--NUT space, let $\gamma$ be an $S^1$-invariant curve in $(X,g)$ and let $p:=\pi(\gamma)$ be the associated point in $U$. Then, $$\Length_X (\gamma)=\frac{2\pi}{\sqrt{\phi(p)}}.$$ Moreover, $\gamma$ is a closed geodesic if and only if $p$ is a critical point of $\phi$
\end{lemma}
\begin{proof}
It follows as in Lemma \ref{S1invarianthypersurfaces}.
\end{proof}

\begin{remark}
Observe that Lemma \ref{S1invarianthypersurfaces}, Lemma \ref{S1invariantsurfaces} and Lemma \ref{S1invariantgeodesics} agree with \cite[Theorem 2]{HsLa71}.
\end{remark}

Lotay and Oliveira proved an existence result for circle-invariant geodesics.
\begin{proposition}[Lotay and Oliveira {\cite{LO20}}] \label{Circle-invariant geodesics}
Let $(X,g)$ be a multi-Eguchi--Hanson or multi-Taub--NUT space with $k\geq2$ singular points of $\phi$. Then, there are at least $k-1$ $S^1$-invariant closed geodesics. Moreover, each such geodesic $\gamma$ is unstable and $\pi(\gamma)$ is contained in the convex hull of the singular points of $\phi$. 
\end{proposition}

The tri-Hamiltonian circle-action, together with Lemma \ref{S1invariantgeodesics}, implies the following result.
\begin{lemma}\label{Lemma ss geodesic}
Let $(X,g)$ be a space given by the Gibbons--Hawking ansatz associated to the harmonic function $\phi$ on $U\subset\mathbb{R}^3$. Then, there are no closed strongly stable geodesics in $X$. 
\end{lemma}
\begin{proof}
	As $\phi$ is harmonic, the strong maximum principle implies that there are no local minima of $\phi$. Hence, Lemma \ref{S1invariantgeodesics} implies that strongly stable geodesics can not be circle-invariant. However, as the tri-Hamiltonian action preserves the length of curves, we deduce that the nullity of a geodesic is positive if the curve is not circle-invariant. We conclude that strongly stable geodesic do not exist.
\end{proof}

\begin{remark} 
Lemma \ref{Lemma ss geodesic} is not true for stable geodesics. Indeed, Oliveira \cite{Oli} found infinitely many examples of spaces given by the Gibbons--Hawking ansatz with a closed stable geodesic. 
\end{remark}

\begin{remark} \label{Remarkonlystrongstabilityaresurfaces}
It follows that the only compact strongly stable submanifolds need to be surfaces. Moreover, similarly as above, we know that a compact strongly stable submanifolds need to be a circle-invariant line connecting two singular points of $\phi$. 
\end{remark}

\section{Barriers For Minimal Submanifolds in the Gibbons--Hawking ansatz}

\subsection{Geometry of circle-invariant hypersurfaces}
Let $(X,g)$ be a space given by the Gibbons--Hawking ansatz associated to the harmonic function $\phi$ on $U\subset\mathbb{R}^3$.

We now relate the second fundamental form of a circle-invariant hypersurface in $X$ to the second fundamental form of its projection in $U$. 
\begin{lemma} \label{SFF S^1 invariant}
Let $N$ be an $S^1$-invariant hypersurface in $(X,g)$ and let $\Sigma:=\pi(N)$ be the associated surface in $U$. Let $(u,v)$ be a local orthonormal frame for $\Sigma\subset U$ with respect to the Euclidean metric and let $\nu$ be an Euclidean unit normal to $\Sigma$. Then, the second fundamental form of $N$ with respect to the unit normal $\tilde{\nu}:=\phi^{-1/2} \nu$, which we denote by $\Gemini^X_{\tilde{\nu}}$, has the form: 
\begin{align*}
    &\Gemini^X_{\tilde{\nu}}(e_0,e_0)=(2\phi)^{-1} \la \nabla_{\mathbb{R}^3} \phi,\tilde{\nu}\ra_{\mathbb{R}^3};
     \hspace{34pt}\\
    &\Gemini^X_{\tilde{\nu}} (\tilde{u},\tilde{u})=\phi^{-1/2} \Gemini^{\mathbb{R}^3}_{\nu} (u,u)-(2\phi)^{-1} \la \nabla_{\R^3} \phi,\tilde{\nu}\ra_{\mathbb{R}^3};\\
    &\Gemini^X_{\tilde{\nu}} (\tilde{u},\tilde{v})=\phi^{-1/2} \Gemini^{\mathbb{R}^3}_{\nu} (u,v);
    \hspace{56pt}\\
    &\Gemini^X_{\tilde{\nu}} (\tilde{v},\tilde{v})=\phi^{-1/2} \Gemini^{\mathbb{R}^3}_{\nu} (v,v)-(2\phi)^{-1} \la \nabla_{\mathbb{R}^3} \phi,\tilde{\nu}\ra_{\mathbb{R}^3};\\
    &\Gemini^X_{\tilde{\nu}} (e_0,\tilde{u})=-(2\phi)^{-1} \la u \times \nabla_{\mathbb{R}^3} \phi, \tilde{\nu}\ra_{\mathbb{R}^3};\hspace{5pt}\\
    &\Gemini^X_{\tilde{\nu}} (e_0,\tilde{v})=-(2\phi)^{-1} \la v \times \nabla_{\mathbb{R}^3} \phi, \tilde{\nu}\ra_{\mathbb{R}^3},
\end{align*}
where $(e_0,\tilde{u}, \tilde{v}):=(\phi^{1/2}\xi,\phi^{-1/2}u,\phi^{-1/2}v)$ is the $S^1$-invariant associated orthonormal frame of $N$ and $\Gemini^{\mathbb{R}^3}_{\nu}$ is the Euclidean second fundamental form of $\Sigma\subset U$ with respect to $\nu$.
\end{lemma}
\begin{proof}
Since $(u,v)$ is a local orthonormal frame for $\Sigma\subset U$ with respect to the Euclidean metric, we can write $u=\sum_{i=1}^3 u_i\partial_i$ and $v=\sum_{i=1}^3 v_i\partial_i$, satisfying $\sum_{i=1}^3 u_i^2=1$, $\sum_{i=1}^3 v_i^2=1$ and $\sum_{i=1}^3 u_i v_i=0$. Denoting by $e_i:=\phi^{-{1/2}} \partial_i$, we use Lemma \ref{Structure} to carry out explicitly the following computations:{\small
\begin{align*}
      &\Gemini^X_{\tilde{\nu}}(e_0,e_0)=g\hspace{-3pt}\left(\frac{1}{2\phi^{3/2}}\sum_{i=1}^3 \frac{\partial\phi}{\partial x_i}e_i,\tilde{\nu}\right)=\frac{1}{2\phi^2}g\left(\nabla_{\mathbb{R}^3} \phi,\tilde{\nu}\right)\hspace{-2pt};\\
      &\Gemini^X_{\tilde{\nu}} (\tilde{u},\tilde{u})=g\hspace{-3pt}\left(\sum_{i,j=1}^3 \phi^{-1} u_i \frac{\partial u_j}{\partial x_i} \partial_j +\frac{1}{2\phi^{3/2}} \sum_{i,j=1}^3 \frac{\partial\phi}{\partial x_j} u_i u_j e_i-\frac{1}{2\phi^{3/2}}\sum_{i,k=1}^3 \frac{\partial\phi}{\partial x_k} u_i^2 e_k,\tilde{\nu}   \right)\hspace{-2pt};\\
      &\Gemini^X_{\tilde{\nu}} (\tilde{v},\tilde{v})=g\hspace{-3pt}\left(\sum_{i,j=1}^3\phi^{-1} v_i \frac{\partial v_j}{\partial x_i} \partial_j +\frac{1}{2\phi^{3/2}} \sum_{i,j=1}^3 \frac{\partial\phi}{\partial x_j} v_i v_j e_i-\frac{1}{2\phi^{3/2}}\sum_{i,k=1}^3 \frac{\partial\phi}{\partial x_k} v_i^2 e_k,\tilde{\nu}   \right)\hspace{-2pt};\\
      &\Gemini^X_{\tilde{\nu}} (\tilde{u},\tilde{v})=g\hspace{-3pt}\left(\sum_{i,j=1}^3 \phi^{-1} u_i \frac{\partial v_j}{\partial x_i} \partial_j +\frac{1}{2\phi^{3/2}} \sum_{i,j=1}^3 \frac{\partial\phi}{\partial x_j} u_i v_j e_i-\frac{1}{2\phi^{3/2}}\sum_{i,k=1}^3 \frac{\partial\phi}{\partial x_k} u_i v_i e_k,\tilde{\nu}   \right)\hspace{-2pt};\\
      &\Gemini^X_{\tilde{\nu}} (e_0,\tilde{u})=-\frac{1}{2\phi^{3/2}} g\hspace{-3pt}\left(\sum_{i,j,k=1}^3 \epsilon_{ijk} u_i \frac{\partial\phi}{\partial x_j} e_k,\tilde{\nu} \right)\hspace{-2pt};\\
      &\Gemini^X_{\tilde{\nu}} (e_0,\tilde{v})=-\frac{1}{2\phi^{3/2}} g\hspace{-3pt}\left(\sum_{i,j,k=1}^3 \epsilon_{ijk} v_i \frac{\partial\phi}{\partial x_j} e_k,\tilde{\nu} \right)\hspace{-2pt},
\end{align*}
}%
where we only used the definition of second fundamental form and that $e_0$ is $g$-normal to $\tilde{\nu}$. Observe that $\Gemini^{\mathbb{R}^{3}}_{\nu}(u,u)=\sum_{i,j}\la u_i \partial_i(u_j) \partial_j,\nu\ra_{\mathbb{R}^3}$, and that analogous formulas hold for $\Gemini^{\mathbb{R}^{3}}_{\nu}(u,v)$ and $\Gemini^{\mathbb{R}^{3}}_{\nu}(v,v)$. 

Clearly, we can write the second fundamental form in the following way: 
\begin{align*}
    &\Gemini^X_{\tilde{\nu}}(e_0,e_0)=(2\phi)^{-1} \la \nabla_{\mathbb{R}^3} \phi,\tilde{\nu}\ra_{\mathbb{R}^3};\\
    &\Gemini^X_{\tilde{\nu}} (\tilde{u},\tilde{u})=\phi^{-1/2} \Gemini^{\mathbb{R}^3}_{\nu} (u,u)+(2\phi)^{-1} \la u(\phi)u -\nabla_{\mathbb{R}^3} \phi,\tilde{\nu}\ra_{\mathbb{R}^3};\\
    &\Gemini^X_{\tilde{\nu}} (\tilde{v},\tilde{v})=\phi^{-1/2} \Gemini^{\mathbb{R}^3}_{\nu} (v,v)+(2\phi)^{-1} \la v(\phi)v -\nabla_{\mathbb{R}^3} \phi,\tilde{\nu}\ra_{\mathbb{R}^3};\\
    &\Gemini^X_{\tilde{\nu}} (\tilde{u},\tilde{v})=\phi^{-1/2} \Gemini^{\mathbb{R}^3}_{\nu} (u,v)+(2\phi)^{-1} \la v(\phi)u,\tilde{\nu}\ra_{\mathbb{R}^3};\\
    &\Gemini^X_{\tilde{\nu}} (e_0,\tilde{u})=-(2\phi)^{-1} \la u \times \nabla_{\mathbb{R}^3} \phi, \tilde{\nu}\ra_{\mathbb{R}^3};\\
    &\Gemini^X_{\tilde{\nu}} (e_0,\tilde{v})=-(2\phi)^{-1} \la v \times \nabla_{\mathbb{R}^3} \phi, \tilde{\nu}\ra_{\mathbb{R}^3},
\end{align*}
which yields the desired equations as $u\perp \tilde{\nu}$ and $v\perp \tilde{\nu}$.
\end{proof}

Analogously, we compute the mean curvature of a circle-invariant hypersurface of $X$ in term of the mean curvature of its projection. 
\begin{lemma} \label{mean curvature S^1 invariant}
Let $N$ be an $S^1$-invariant hypersurface in $(X,g)$ and let $\Sigma:=\pi(N)$ be the associated surface in $U$. Then, the mean curvature of $N$, which we denote by $\H^X_{N}$, has the form:
\begin{align*}
    \H^X_{N}&=-\frac{1}{2\phi^2}\nabla^\perp_{\mathbb{R}^3}\phi+\frac{1}{\phi}\H^{\mathbb{R}^3}_{\Sigma}\\
&=\frac{1}{\phi} (\nabla^\perp_{\mathbb{R}^3} \log \phi^{-1/2} +\H^{\mathbb{R}^3}_{\Sigma}),
\end{align*}
where $\H^{\mathbb{R}^3}_{\Sigma}$ is the Euclidean mean curvature of $\Sigma\subset U$.
\end{lemma}

\begin{proof}
Let $(u,v)$ be a local orthonormal frame for $\Sigma\subset U$ with respect to the Euclidean metric, i.e. $u=\sum_{i=1}^3 u_i\partial_i$, $v=\sum_{i=1}^3 v_i\partial_i$ satisfying $\sum_{i=1}^3 u_i^2=1$, $\sum_{i=1}^3 v_i^2=1$ and $\sum_{i=1}^3 u_i v_i=0$. A local orthonormal frame for $N$ is $(e_0,\tilde{u}, \tilde{v}):=(\phi^{1/2}\xi,\phi^{-1/2}u,\phi^{-1/2}v)$. We now compute the mean curvature of $N$, using Lemma \ref{Structure}, as follows:
\begin{align*}
   \H^{X}_N&=(\nabla_{e_0} e_0+\nabla_{\tilde{u}} \tilde{u}+\nabla_{\tilde{v}}\tilde{v})^{\perp}\\
   &=\frac{1}{2\phi^{3/2}}\left(\sum_{i=1}^3 \frac{\partial \phi}{\partial x_i} e_i+\sum_{i,j=1}^3 \frac{\partial \phi}{\partial x_j}e_i(u_i u_j+v_i v_j)-\sum_{i,j=1}^3 \frac{\partial \phi}{\partial x_j} e_j (u_i^2+v_i^2)\right)^\perp \\
   &\hspace{13pt}+\frac{1}{\phi} \H^{\mathbb{R}^3}_\Sigma\\
  &=\frac{1}{2\phi^{3/2}}\left(-\sum_{i=1}^3 \frac{\partial \phi}{\partial x_i} e_i +\sum_{i,j=1}^3 \frac{\partial \phi}{\partial x_j}e_i(u_i u_j+v_i v_j) \right)^{\perp}+\frac{1}{\phi} \H^{\mathbb{R}^3}_\Sigma\\
   &=\frac{1}{2\phi^{2}} \left(-{\nabla}_{\mathbb{R}^3} \phi +{\nabla^T_{\mathbb{R}^3}} \phi \right)^{\perp} +\frac{1}{\phi} \H^{\mathbb{R}^3}_\Sigma\\
   &=-\frac{1}{2\phi^2}\nabla^\perp_{\mathbb{R}^3}\phi+\frac{1}{\phi}\H^{\mathbb{R}^3}_{\Sigma}, 
\end{align*}
where $e_i=\phi^{-1/2}\partial_i$. 
\end{proof}

\begin{remark}
Observe that this result agrees with Corollary \ref{S1invarianthypersurfaces} and Lemma \ref{SFF S^1 invariant}. 

Indeed, if we denote by $\tilde{\H}_{\Sigma}$ the mean curvature of $\Sigma$ in $(U,\phi^{1/2}g_{\mathbb{R}^3})$, then the following equation:
\[
\phi^{1/2}\tilde{\H}_{\Sigma}=\H^{\mathbb{R}^3}_{\Sigma}+\nabla^\perp_{\mathbb{R}^3} \log \phi^{-1/2}, 
\]
is precisely the formula for the mean curvature under conformal change of metric. 
The other claim is obvious.
\end{remark}

\subsection{Barriers for minimal hypersurfaces} In the multi-Eguchi--Hanson and multi-Taub-- NUT spaces, we can use a barrier argument to prove that there are no compact minimal hypersurfaces outside certain regions. If we choose the points 
lying on a line, then, this argument is enough to prove the non-existence of compact minimal hypersurfaces. 

\begin{lemma} \label{SphereBarrierHypLemma}
Let $(X,g)$ be a multi-Eguchi--Hanson or a multi-Taub--NUT space and let $N_r$ be the $S^1$-invariant hypersurface in $X$ corresponding to the Euclidean sphere $S_r (0):=\{x\in U: \av{x}_{\mathbb{R}^3}=r\}\subset U$, i.e. $\pi(N_r)=S_r(0)$ for some $r\in\mathbb{R}^+ \setminus\{\av{p_i}_{\mathbb{R}^3}\}_{i=1}^k$. 
Then, $N_r$ is strictly $3$-convex with respect to the interior of the sphere for all $r> 4/3 \max_i \av{p_i}_{\mathbb{R}^3}$ and all $r<\min \{\av{p_i}_{\mathbb{R}^3}: \av{p_i}_{\mathbb{R}^3}>0\}$.  
\end{lemma}
\begin{proof}
Since we know that the mean curvature of $S_r(0)\subset \mathbb{R}^3$ is $-\frac{2x}{\av{x}^2_{\mathbb{R}^3}}$, we can use Lemma \ref{mean curvature S^1 invariant} to compute the mean curvature of $N_r$:
\begin{align*}
   \H^X_{N_r}&=-\frac{1}{\phi}\left(\frac{1}{2\phi} \la \nabla\phi,x\ra \frac{x}{\av{x}^2}+\frac{2x}{\av{x}^2}\right)\\
    &=-\frac{1}{2\phi^2} \frac{x}{\av{x}^2} \left(\la \nabla\phi,x\ra +4\phi \right),
\end{align*}
where $\la.,.\ra$, $\av{\hspace{1pt}.\hspace{1pt}}$ and $\nabla$ are with respect to the Euclidean metric. Since $\phi$ is positive, it's enough to show that $\la \nabla\phi,x\ra +4\phi>0$ everywhere. 
Explicitly we compute:
\begin{align*}
    \la \nabla\phi,x\ra +4\phi &\geq -\frac{1}{2}\sum_{i=1}^k \frac{\la x-p_i,x\ra}{\av{x-p_i}^3}+4\sum_{i=1}^k \frac{1}{2\av{x-p_i}}\\
    &=\sum_{i=1}^k \frac{1}{2\av{x-p_i}^3} \left( 4\av{x-p_i}^2-\la x-p_i,x\ra \right)\\
       &=\sum_{i=1}^k \frac{1}{2\av{x-p_i}^3} \left( 4(\av{x}^2-2\la x, p_i\ra +\av{p_i}^2)-(\av{x}^2-\la {p_i}, x\ra) \right)\\
    &=\sum_{i=1}^k \frac{1}{2\av{x-p_i}^3} \left( 3\av{x}^2-7\la x,p_i\ra +4\av{p_i}^2\right)\\
    &\geq \sum_{i=1}^k \frac{1}{2\av{x-p_i}^3} \left( 3\av{x}-4\av{p_i}\right)(\av{x}-\av{p_i}),
\end{align*}
where we used $m\geq0$ and Cauchy--Schwarz inequality. Observe that, if $\av{p_i}=0$ for some $i$, then the related summand in the last line is automatically positive for all $\av{x}$. Under the conditions on $\av{x}=r$, it is clear that $\la \nabla\phi,x\ra +4\phi>0$ and hence we conclude.
\end{proof}

\begin{thm} \label{SphereBarrierHyp}
Let $(X,g)$ be a multi-Eguchi--Hanson or a multi-Taub--NUT space. Then, compactly supported stationary integral 3-varifolds need to be contained in $\pi^{-1} (\{x\in U: \av{x}_{\mathbb{R}^3}\leq4/3 \max_i \av{p_i}_{\mathbb{R}^3}\})$. Moreover, there are no compactly supported stationary integral 3-varifolds which are contained in $\pi^{-1} \left(\left\{x\in U: \av{x}_{\mathbb{R}^3}< \min \{\av{p_i}_{\mathbb{R}^3}: \av{p_i}_{\mathbb{R}^3}>0\}\right\}\right)$.
\end{thm}
\begin{proof}
Assume by contradiction that there is a compactly supported stationary integral 3-varifolds $T$ which is not contained in $\pi^{-1} (\{x\in U: \av{x}_{\mathbb{R}^3}\leq4/3 \max_i \av{p_i}_{\mathbb{R}^3}\})$. By assumption, there exists an $r>4/3 \max_i \av{p_i}_{\mathbb{R}^3}$ such that $T$ is supported in the interior of $N_r$ and the support of $T$ intersects $N_r$. $N_r$ is the $S^1$-invariant hypersurface corresponding to the Euclidean sphere $S_r (0)$. Observing that $N_r$ is strictly 3-convex by Lemma \ref{SphereBarrierHypLemma}, we get a contradiction to Corollary \ref{BarrierGMT}. A similar argument works for the second part of the statement.
\end{proof}
Observe that Theorem \ref{IntroHypersurfaceSphere} is the special case of Thoerem \ref{SphereBarrierHyp} in the smooth setting. If we also assume orientability, the proof can be simplified by using Corollary \ref{BarrierSmooth} instead of Corollary \ref{BarrierGMT}.

\begin{figure} [t] 

\begin{tikzpicture}[scale=0.7]
\draw [-,thick, green, pattern=north west lines, pattern color=green!40] (-3,0) to [out=90,in=160] (-1,0.8)
to [out=-20,in=180] (0,0.6) to [out=0, in=90] (1.4,0)
to [out=-90, in=45] (-0.2,-0.4) to [out=225, in=-90] (-3,0);
\draw[black,thick] (0,0) circle (2 cm);
\draw[black,thick] (-2,0) arc (180:360:2 cm and 0.4cm);
\draw[black, thick, dashed] (2,0) arc (0:180:2 cm and 0.4cm);
\fill[black] (-1,0) circle (0.05cm) node[black,left] {$p_1$};
\fill[black] (1.2,+0.1) circle (0.05cm) node[black,right] {$p_2$};
\fill[black] (0.6,-0.1) circle (0.05cm) node[black,right] {$p_3$};
\fill[black] (0,0) circle (0.05cm) node[black,left] {$O$};
\draw[dashed, thick,red, rotate=+47] (0,0)--(2,0);
\node[right, red] at (1.22,1.65) {$4/3 \max \lvert p_i \rvert$};

\draw[blue, rotate=37] (3,0) arc (0:352:3 cm);
\draw[blue] (-3,0) arc (180:360:3 cm and 0.9cm);
\draw[blue, dashed] (3,0) arc (0:180:3 cm and 0.9cm);
\draw[blue, thick,->] (-3,0)--(-2.2,0); 
\node[blue] at (-2.6,0.24) {$\H$};
\end{tikzpicture}
\caption{Spherical barriers used in Theorem \ref{SphereBarrierHyp}.}
\end{figure}
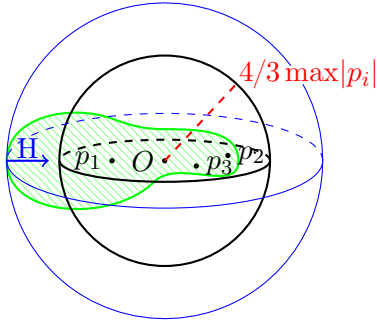

\begin{lemma} \label{CylinderBarrierHypLemma}
Let $(X,g)$ be a multi-Eguchi--Hanson or a multi-Taub--NUT space and let $N_r$ be the $S^1$-invariant hypersurface in $X$ corresponding to the Euclidean cylinder $\Sigma_r:=\{x\in U: x_1^2+x_2^2=r^2\}\subset U$, i.e. $\pi(N_r)=\Sigma_r$ for some $r\in\mathbb{R}^+\setminus\{r_i\}_{i=1}^k$, where $r_i:=\sqrt{(p_i)_1^2+(p_i)_2^2}$. 
Then, $N_r$ is strictly $3$-convex with respect to the interior of the cylinder for all $r> 2 \max_i r_i$ and all $r<\min \{r_i: r_i >0\}$.

\end{lemma}
\begin{proof}
As above, since we know that the mean curvature of $\Sigma_r\subset \mathbb{R}^3$ at $x=(r\cos \theta,r\sin \theta,x_3)$ is $-\frac{\nu}{r}$, where $\nu=(\cos \theta,\sin \theta,0)$, we can use Lemma \ref{mean curvature S^1 invariant} to compute the mean curvature of $N_r$:
\begin{align*}
    \H^X_{N_r}&=-\frac{1}{\phi}\left(\frac{1}{2\phi} \la \nabla\phi,\nu\ra \nu+\frac{\nu}{r} \right)\\
    &=-\frac{1}{2\phi^2} \nu \left(\la \nabla\phi,\nu \ra +\frac{2\phi}{r} \right),
\end{align*}
where $\la.,.\ra$, $\av{\hspace{1pt}.\hspace{1pt}}$ and $\nabla$ are with respect to the Euclidean metric. Clearly, it suffices to show that $\la \nabla\phi,\nu \ra +\frac{2\phi}{r}>0$ everywhere. As in Lemma \ref{SphereBarrierHypLemma}:
\begin{align*}
    \la \nabla\phi,\nu \ra +\frac{2\phi}{r} &\geq -\frac{1}{2}\sum_{i=1}^k \frac{\la x-p_i,\nu \ra}{\av{x-p_i}^3}+\frac{1}{r}\sum_{i=1}^k \frac{1}{\av{x-p_i}}\\
    &=\frac{1}{2r} \sum_{i=1}^k \frac{1}{\av{x-p_i}^3} \left( 2\av{x-p_i}^2-r\la x-p_i,\nu \ra \right)\\
    &\geq\frac{1}{2r} \sum_{i=1}^k \frac{1}{\av{x-p_i}^3} \left( 2r^2+2r_i^2+2((x)_3-(p_i)_3)^2 -3rr_i-r^2 \right)\\
    &\geq\frac{1}{2r} \sum_{i=1}^k \frac{1}{\av{x-p_i}^3} (r-r_i)(r-2r_i),
\end{align*}
where we used $m\geq0$, Cauchy--Schwarz inequality and $(x_3-(p_i)_3)^2\geq0$.

We conclude as in Lemma \ref{SphereBarrierHypLemma}.
\end{proof}

\begin{thm} \label{CylinderBarrierHyp}
Let $(X,g)$ be a multi-Eguchi--Hanson or a multi-Taub--NUT space. Then, compactly supported stationary integral 3-varifolds need to be contained in $\pi^{-1} (\{x\in U: \sqrt{x_1^2+x_2^2}\leq 2 \max_i r_i\})$, where $r^2_i=(p_i)_1^2+(p_i)_2^2$. Moreover, there are no compactly supported stationary integral 3-varifolds contained in $\pi^{-1} (\{x\in U: \sqrt{x_1^2+x_2^2}< \min \{r_i: r_i>0\}\})$.
\end{thm}
\begin{proof}
The proof follows almost verbatim Theorem \ref{SphereBarrierHyp}, substituting Lemma \ref{SphereBarrierHypLemma} and the related sets with Lemma \ref{CylinderBarrierHypLemma} and the related sets. 
\end{proof}

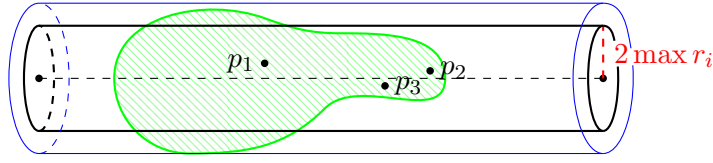
\begin{figure}[t]
\begin{tikzpicture} 
\draw [-,thick, green, pattern=north west lines, pattern color=green!40] (-3,0) to [out=90,in=160] (-1,0.8)
to [out=-20,in=180] (0,0.6) to [out=0, in=90] (1.4,0)
to [out=-90, in=45] (-0.2,-0.4) to [out=225, in=0] (-2,-1) to[out=180, in=-90] (-3,0);
\fill[black] (-1,0.2) circle (0.05cm) node[black,left] {$p_1$};
\fill[black] (1.2,+0.1) circle (0.05cm) node[black,right] {$p_2$};
\fill[black] (0.6,-0.1) circle (0.05cm) node[black,right] {$p_3$};
\fill[black] (3.5,0) circle (0.05cm);
\fill[black] (-4,0) circle (0.05cm);
\draw[dashed] (-4,0) -- (3.5,0);
\draw[red, thick, dashed] (3.5,0) -- (3.5,0.7);
\node at (3.51,0.3) [right, red] {$2 \max r_i$};

\draw[thick] (-4,0.7) arc (90:270:0.2 cm and 0.7cm);
\draw[thick, dashed] (-4,-0.7) arc (-90:90:0.2 cm and 0.7cm);
\draw[thick] (3.5,0.7) arc (90:270:0.2 cm and 0.7cm);
\draw[thick] (3.5,-0.7) arc (-90:13:0.2 cm and 0.7cm);
\draw[thick] (3.65,0.5) to [out=100, in=0] (3.5,0.7);
\draw[thick] (-4,0.7) -- (3.5,0.7);
\draw[thick] (-4,-0.7) -- (3.5,-0.7);

\draw[blue] (3.5,1) arc (90:270:0.4 cm and 1cm);
\draw[blue] (3.5,-1) arc (-90:10:0.4 cm and 1cm);
\draw[blue] (3.87,0.40) to [out=100, in=0] (3.5,1);

\draw[blue] (-4,1) -- (3.5,1);
\draw[blue] (-4,-1) -- (3.5,-1);
\draw[blue] (-4,1) arc (90:270:0.4 cm and 1cm);
\draw[blue,dashed] (-4,-1) arc (-90:90:0.4 cm and 1cm);
\end{tikzpicture}
\caption{Cylindrical barriers used in Theorem \ref{CylinderBarrierHyp}.}
\end{figure}

\begin{remark}
Since rotations and translations of $\mathbb{R}^3$ induce isometric representations of $(X,g)$, we can consider, as barriers, spheres centred in any point of $\mathbb{R}^3$ and cylinders with any axis. Even though we have a lot of barrier sets, these are not enough to prove the global non-existence of compact minimal hypersurfaces in the general case.

However, in the following important case we do have a global non-existence result.
\end{remark}

\begin{corollary} \label{Nonexistence Collinear Case}
Let $(X,g)$ be a multi-Eguchi--Hanson or a multi-Taub--NUT space with the $\{p_i\}_{i=1}^k$ lying on a line. Then, there are no compactly supported stationary integral 3-varifolds in $X$.
\end{corollary}
\begin{proof}
By previous remark, we can chooses $(p_i)_1=0$ and $(p_i)_2=0$ for all $i$. It follows that $r_i=0$ for all $i$, hence, we can conclude by Theorem \ref{CylinderBarrierHyp}.
\end{proof}

As above, Theorem \ref{IntroHyperusrfaceCylinders} and Corollary \ref{IntroCorollaryCollinearCase} are the special cases of Theorem \ref{CylinderBarrierHyp} and Corollary \ref{Nonexistence Collinear Case} in the smooth setting.

\begin{lemma} \label{PlaneBarrierHypLemma}
Let $(X,g)$ be a multi-Eguchi--Hanson or a multi-Taub--NUT space and let $N_r$ be the $S^1$-invariant hypersurface in $X$ corresponding to the Euclidean plane $\Pi_r:=\{x\in U: x_3=r\}\subset U$, i.e. $\pi(N_r)=\Pi_r$ for some $r\in\mathbb{R}\setminus\{(p_i)_3\}_{i=1}^k$. 
Then, for all $r$ such that $r>\max_i (p_i)_3$ or $r<\min_i (p_i)_3$, $N_r$ is strictly $3$-convex with respect to the half-space not containing the singular points of $\phi$.
\end{lemma}
\begin{proof}
The proof, analogously to Lemma \ref{SphereBarrierHypLemma} and Lemma \ref{CylinderBarrierHypLemma}, relies on Lemma \ref{mean curvature S^1 invariant}.
\end{proof}

\begin{thm} \label{PlaneBarrierHyp}
Let $(X,g)$ be a multi-Eguchi--Hanson or a multi-Taub--NUT space. Then, there are no compact minimal hypersurfaces (compactly supported stationary integral 3-varifolds) contained in $\pi^{-1} (\{x\in U: {x_3}>\max_i {(p_i)_3}\})$ or in $\pi^{-1} (\{x\in U: x_3<\min {(p_i)_3}\})$.
\end{thm}
\begin{proof}
As in Theorem \ref{SphereBarrierHyp}, it is an application of Corollary \ref{BarrierGMT} together with Lemma \ref{PlaneBarrierHypLemma}.
\end{proof}

\begin{remark}
It is easy to see that the results in this subsection are still true for multi-centred Gibbons--Hawking spaces.
\end{remark}

\subsection{Barriers for minimal submanifolds of higher codimension}
Similarly to the hypersurface case, in the multi-Eguchi--Hanson and multi-Taub--NUT spaces, we can use a barrier argument to prove that there are no compact minimal submanifolds outside certain regions. 

\begin{lemma} \label{SphereBarrierLemma}
Let $(X,g)$ be a multi-Eguchi--Hanson or a multi-Taub--NUT space and let $N_r$ be the $S^1$-invariant hypersurface in $X$ corresponding to the Euclidean sphere $S_r (0):=\{x\in U: \av{x}_{\mathbb{R}^3}=r\}\subset U$, i.e. $\pi(N_r)=S_r(0)$ for some $r\in\mathbb{R}^+\setminus\{\av{p_i}_{\mathbb{R}^3}\}_{i=1}^k$. Then, $N_r$ is strictly convex with respect to the interior of the sphere for all $r> C \max_i \av{p_i}_{\mathbb{R}^3}$, where $C\approx 5.07$ is the only real root of the polynomial: $-x^3+4x^2+5x+2$. Moreover, if $p_i=0$ for some $i$, then, there exists an $r_0$ small enough such that $N_r$ is strictly convex for all $r<r_0$.
\end{lemma}
\begin{proof}
Let $\nu:=-x/\av{x}_{\mathbb{R}^3}$ be the unit normal for $S_r(0)\subset U$ pointing inward. We recall that, with respect to $\nu$, the second fundamental form of $S_r(0)$ is: 
$$\Gemini_{\nu}^{\mathbb{R}^3}(u,v)=\frac{1}{r} \la u,v\ra_{S_r(0)},$$ 
for all $u,v$ tangent vectors of $S_r(0)$. 

Given any $(u,v)$ local orthonormal frame for $S_r(0)$, Lemma \ref{SFF S^1 invariant} implies that the second fundamental form of $N_r$ with respect to $\tilde{\nu}:=\phi^{-1/2} \nu$, in the basis $(\phi^{-1/2} u, \phi^{-1/2} v, \phi^{1/2}\xi)$, can be written as the matrix:
$$\setlength{\arraycolsep}{1pt}
  \renewcommand{\arraystretch}{1.4}
  \small\begin{bmatrix}
\phi^{-1/2}\frac{1}{r}-(2\phi)^{-1} \la \nabla_{\mathbb{R}^3} \phi,\tilde{\nu}\ra & 0 & -(2\phi)^{-1} \la u \times \nabla_{\mathbb{R}^3} \phi, \tilde{\nu}\ra\\
0 & \phi^{-1/2}\frac{1}{r}-(2\phi)^{-1} \la \nabla_{\mathbb{R}^3} \phi,\tilde{\nu}\ra & -(2\phi)^{-1} \la v \times \nabla_{\mathbb{R}^3} \phi, \tilde{\nu}\ra\\
-(2\phi)^{-1} \la u \times \nabla_{\mathbb{R}^3} \phi, \tilde{\nu}\ra & -(2\phi)^{-1}\la v \times \nabla_{\mathbb{R}^3} \phi, \tilde{\nu}\ra & (2\phi)^{-1} \la \nabla_{\R^3} \phi,\tilde{\nu}\ra
\end{bmatrix}.
$$
Note that, if it satisfies Sylvester's criterion everywhere, we have that $N_r$ is strictly convex. 

The first two minors are positive if and only if $\phi^{-1/2}\frac{1}{r}-(2\phi)^{-1} \la \nabla_{\mathbb{R}^3} \phi,\tilde{\nu}\ra_{\mathbb{R}^3}>0$ or, equivalently, $\la \nabla_{\mathbb{R}^3} \phi,x\ra_{\mathbb{R}^3} +2\phi>0$. In a similar fashion to Lemma \ref{CylinderBarrierHypLemma}, we compute:
\begin{align*}
\la \nabla_{\mathbb{R}^3} \phi,x\ra_{\mathbb{R}^3} +2\phi & \geq \sum_{i=1}^k \frac{1}{2\av{x-p_i}^3_{\mathbb{R}^3}}\left(\av{x}_{\mathbb{R}^3}^2-3\av{x}_{\mathbb{R}^3} \av{p_i}_{\mathbb{R}^3}+2\av{p_i}^2_{\mathbb{R}^3}\right)\\
&=\sum_{i=1}^k \frac{1}{2\av{x-p_i}^3}\left(\av{x}_{\mathbb{R}^3}-2\av{p_i}_{\mathbb{R}^3}\right)\left(\av{x}_{\mathbb{R}^3}-\av{p_i}_{\mathbb{R}^3}\right).
\end{align*}
If $r=\av{x}_{\mathbb{R}^3}>2\max_i \av{p_i}_{\mathbb{R}^3}$ or $r<\min \{\av{p_i}_{\mathbb{R}^3}: \av{p_i}_{\mathbb{R}^3}>0\}$, then this sum is positive. We are left to prove that the determinant of the matrix is positive. An explicit computation shows that $\det (\Gemini^X_{\tilde{\nu}} )$  factors as the product of 
$$\phi^{-1/2}\frac{1}{r}-(2\phi)^{-1} \la \nabla_{\mathbb{R}^3} \phi,\tilde{\nu}\ra_{\mathbb{R}^3}$$
and of
$$\frac{1}{2\phi^2}\left( \frac{\la \nabla_{\mathbb{R}^3} \phi,\nu\ra_{\mathbb{R}^3}}{r}-(2\phi)^{-1} \av{\nabla_{\mathbb{R}^3} \phi}_{\mathbb{R}^3}^2\right),$$
where we used the properties of the cross product and the fact that $(u,v,\nu)$ forms an orthonormal basis of the tangent space of $\mathbb{R}^3$. Since the former is equal to the first minor, we just need to prove that the latter is positive or, equivalently, that 
\[
\av{\nabla_{\mathbb{R}^3} \phi}_{\mathbb{R}^3}^2+\frac{2\phi\la \nabla_{\mathbb{R}^3}\phi, x\ra_{\mathbb{R}^3}}{\av{x}_{\mathbb{R}^3}^2}<0.
\]
Explicitly, the two summands are:
\begin{align} \label{Sum1}
\av{\nabla_{\mathbb{R}^3} \phi}_{\mathbb{R}^3}^2=\frac{1}{4}\sum_{i,j=1}^k \frac{\la x-p_i,x-p_j\ra_{\mathbb{R}^3}}{\av{x-p_i}^3_{\mathbb{R}^3}{\av{x-p_j}}^3_{\mathbb{R}^3}}
\end{align}
and
\begin{align} 
\begin{split}\label{Sum2}
\frac{2\phi\la \nabla_{\mathbb{R}^3}\phi, x\ra_{\mathbb{R}^3}}{\av{x}_{\mathbb{R}^3}^2}&=\frac{1}{\av{x}_{\mathbb{R}^3}^2}\left(\sum_{i=1}^k \frac{1}{\av{x-p_i}_{\mathbb{R}^3}} +2m\right) \left(-\frac{1}{2} \sum_{i=1}^k \frac{\av{x}_{\mathbb{R}^3}^2-\la p_i,x\ra_{\mathbb{R}^3}}{\av{x-p_i}_{\mathbb{R}^3}^3} \right)\\
&\leq -\frac{1}{4\av{x}_{\mathbb{R}^3}^2} \sum_{i,j=1}^k \frac{\av{x}_{\mathbb{R}^3}^2-\la p_i,x\ra_{\mathbb{R}^3}}{\av{x-p_i}^3_{\mathbb{R}^3} \av{x-p_j}_{\mathbb{R}^3}}\\
 &\hspace{12pt}-\frac{1}{4\av{x}_{\mathbb{R}^3}^2} \sum_{i,j=1}^k \frac{\av{x}_{\mathbb{R}^3}^2-\la p_j,x\ra_{\mathbb{R}^3}}{\av{x-p_j}^3_{\mathbb{R}^3} \av{x-p_i}_{\mathbb{R}^3}},
\end{split}
\end{align}
where inequality holds if $\av{x}_{\mathbb{R}^3}\geq\av{p_i}_{\mathbb{R}^3}$ for all $i$. 

Summing Equation \ref{Sum1} and Equation \ref{Sum2}, we obtain: 
\begin{align*}
    \av{\nabla \phi}^2+\frac{2\phi\la \nabla\phi, x\ra}{\av{x}^2}\leq &\frac{1}{4\av{x}^2}\sum_{i,j=1}^k \bigg(\frac{\la x-p_i,x-p_j\ra\av{x}^2 -\av{x}^2\av{x-p_j}^2}{\av{x-p_j}^3\av{x-p_i}^3}\\
    &+\frac{\la p_i,x\ra \av{x-p_j}^2 -\av{x}^2\av{x-p_i}^2 +\la p_j,x\ra \av{x-p_i}^2 }{\av{x-p_j}^3\av{x-p_i}^3}\bigg).
\end{align*}
Let's denote by $(I)$ the numerator of such expression and by $A:=\max_i \av{p_i}$. We can write:
\begin{align*}
(I)&=-\av{x}^4+\la p_i,p_j\ra \av{x}^2 -4\la p_i,x\ra \la p_j,x\ra-\av{x}^2 (\av{p_j}^2+\av{p_i}^2)\\
&\hspace{13pt}+2\av{x}^2(\la x,p_j\ra+\la x,p_i\ra)+\av{p_j}^2 \la p_i,x\ra +\av{p_i}^2\la p_j,x\ra\\
&\leq -\av{x}^4 +\av{p_i}\av{p_j}\av{x}^2+4\av{p_i}\av{p_j}\av{x}^2-\av{x}^2\av{p_j}^2-\av{x}^2\av{p_i}^2\\
&\hspace{13pt} +2\av{x}^3\av{p_j}+ 2\av{x}^3\av{p_i} +\av{p_j}^2\av{p_i}\av{x}+\av{p_i}^2\av{p_j}\av{x}\\
&\leq -\av{x}^4 +5\av{p_i}\av{p_j}\av{x}^2+ 2\av{x}^3(\av{p_j}+\av{p_i})+\av{p_j}^2 \av{p_i}\av{x}+\av{p_i}^2\av{p_j}\av{x}\\
&\leq \av{x}\left( -\av{x}^3 + 4A\av{x}^2 +5A^2 \av{x} +2A^3\right), 
\end{align*}
where we only developed $(I)$, applied Cauchy-Schwarz and used the obvious estimate $\av{p_i}\leq A$. The first claim follows immediately. 

We will now deal with the second part of the statement. Without loss of generality, we can assume that $p_1=0$. Considering the expression of $\av{\nabla_{\mathbb{R}^3} \phi}_{\mathbb{R}^3}^2$, we can distinguish 3 different cases: 
\begin{enumerate}
    \item $i=j=1$;
    \item $i=1$ and $j\neq 1$;
    \item $i,j\neq1$.
\end{enumerate}

Under the assumption that $r$ is small enough, we can estimate all the terms in $(3)$ with a constant, all the terms in $(2)$ with a constant times $1/\av{x}^2$ and the term in $(1)$ with $1/(4\av{x}^4)$. Hence, we have: 
\begin{align*}
    \av{\nabla_{\mathbb{R}^3} \phi}_{\mathbb{R}^3}^2\leq \frac{1}{4\av{x}^4}+\frac{B_1}{\av{x}^2}+B_2.
\end{align*}
Analogously, we can estimate: 
\begin{align*}
    \frac{2\phi\la \nabla_{\mathbb{R}^3}\phi, x\ra_{\mathbb{R}^3}}{\av{x}_{\mathbb{R}^3}^2}\leq -\frac{1}{2\av{x}^4}+\frac{C_1}{\av{x}^3}+\frac{C_2}{\av{x}^2}.
\end{align*}
It is clear that, for $\av{x}$ small enough, the following holds everywhere: 
\[
\av{\nabla_{\mathbb{R}^3} \phi}_{\mathbb{R}^3}^2+\frac{2\phi\la \nabla_{\mathbb{R}^3}\phi, x\ra_{\mathbb{R}^3}}{\av{x}_{\mathbb{R}^3}^2}<0.
\]
Thus, the proof is complete.
\end{proof}

\begin{thm} \label{SphereBarrier}
Let $(X,g)$ be a multi-Eguchi--Hanson or a multi-Taub--NUT space. Then, compactly supported stationary integral varifolds need to be contained $\pi^{-1} (\{x\in U: \av{x}\leq C \max_i\av{p_i}_{\mathbb{R}^3}\})$, where $C\approx 5.07$ is the only real root of the polynomial: $-x^3+4x^2+5x+2$. Moreover, if $p_i=0$ for some $i$, then, there are no compactly supported stationary integral varifolds contained in $\pi^{-1} \left(\left\{x\in U: \av{x}< r_0\right\}\right)$, for some $r_0$ small enough. 
\end{thm}
\begin{proof}
It follows as in Theorem \ref{SphereBarrierHyp}, substituting Lemma \ref{SphereBarrierHypLemma} with Lemma \ref{SphereBarrierLemma}. 
\end{proof}

Once again, note that Theorem \ref{IntroHigherCodimensionSpheres} is the smooth special case of Theorem \ref{SphereBarrier}. Moreover, we can consider Theorem \ref{SphereBarrier} as a generalization of the following classical result.
\begin{corollary} \label{ClassicalCorollaryR^4Taub-NUT}
There are no compact minimal submanifolds (compactly supported stationary integral varifolds) in the Euclidean $\mathbb{R}^4$ and in the Taub--NUT space. 
\end{corollary}

\begin{remark}
Differently from the codimension 1 case, we observe that it is not possible to carry out a similar argument with cylinders and planes. Indeed, cylinders correspond to hypersurfaces that are nowhere convex. The reason is that cylinders in $\mathbb{R}^3$ have one vanishing principal curvature. Hence, using Lemma \ref{SFF S^1 invariant} with the principal directions as a basis, it is straightforward to verify that an element of the diagonal of the second fundamental form is less than or equal to zero. Obviously, Sylvester's criterion cannot hold. Analogously, the same argument works for planes.

Moreover, if the points are collinear, cylinders with axis containing the singular points of $\phi$ correspond to hypersurfaces that are nowhere $2$-convex. Indeed, in the same setting as above, the second fundamental form is simple enough that it is possible to explicitly compute its eigenvalues. It is easy to see that the sum of two of them is always less than zero. Now, consider a plane orthogonal to the line containing the singular points of $\phi$. It is easy to see that, if all the points are contained in one of its half-spaces, the corresponding matrix at the point of intersection with the line is diagonal. Since the sum of the smallest entries is zero, we conclude that it cannot be strictly two-convex.

In particular, we have shown that, even for weaker constants, Theorem \ref{CylinderBarrierHyp} and Theorem \ref{PlaneBarrierHyp} cannot hold in higher codimension.
\end{remark}

Finally, we generalize Theorem \ref{TheoremnonexistenceEH} to the two-centred multi-Taub--NUT space.

\begin{lemma} \label{EllipsoidBarrierLemma}
Let $(X,g)$ be a multi-Eguchi--Hanson or a multi-Taub--NUT space with two singular points of $\phi$, which, without loss of generality,  we can assume to be $p_{\pm}:=(0,0,\pm a)$, and let $N_r$ be the $S^1$-invariant hypersurface corresponding to the Euclidean ellipsoid $\Sigma_r=\{x\in U: \av{x-p_+}_{\mathbb{R}^3}+\av{x-p_-}_{\mathbb{R}^3}=2a\cosh r\}\subset U$, i.e. $\pi(N_r)=\Sigma_r$ for some $r\in\mathbb{R}^{+}$. Then, $N_r$ is strictly convex with respect to the interior of the ellipsoid for all $r>0$.
\end{lemma}

\begin{proof}
Given the parametrization of $\Sigma_r$, $r>0$: 
\[
\systeme*{x_1=&a\sinh r \sin \beta\cos \alpha, x_2=&a\sinh r \sin \beta \sin \alpha, x_3=&a\cosh r \cos \beta}
\]
for $\alpha\in[0,2\pi)$ and $\beta\in[0,\pi]$, we observe that $u:=\partial_\alpha/\av{\partial \alpha}$ and $v:=\partial_\beta/\av{\partial_\beta}$ form an orthonormal basis for $\Sigma_r$ and $\nu:=u\times v$ is the inward pointing unit normal. Moreover, we have: 
\begin{align*}
&\Gemini^{\mathbb{R}^3}_\nu (u,u)=\frac{\cosh r}{aA \sinh r}; \hspace{55pt} \Gemini^{\mathbb{R}^3}_\nu (u,v)\equiv 0;\\
&\Gemini^{\mathbb{R}^3}_\nu (v,v)=\frac{\sinh r\cosh r}{aA^3}; \hspace{42pt} \la \nabla \phi, u\ra_{\mathbb{R}^3}\equiv 0;\\
&\la \nabla \phi, v\ra_{\mathbb{R}^3}=\sum_{\pm} \frac{\mp\sin\beta}{2\av{x-p_{\pm}}^2_{\mathbb{R}^3}A};\hspace{20pt} \la \nabla_{\mathbb{R}^3} {\phi},\nu\ra_{\mathbb{R}^3}=\sum_{\pm} \frac{\sinh r}{2\av{x-p_{\pm}}^2_{\mathbb{R}^3}A},
\end{align*}
where $A^2=(\cosh r-\cos \beta)(\cosh r+\cos\beta)$.

By lemma \ref{SFF S^1 invariant}, in the basis $(\tilde{\phi}^{-1/2} u, \tilde{\phi}^{-1/2} v, \tilde{\phi}^{1/2}\xi)$, the matrix representing the second fundamental form of $N_r$ with respect to $\tilde{\nu}:=\phi^{-1/2}\nu$ is:
$$
\setlength{\arraycolsep}{2pt}
  \renewcommand{\arraystretch}{0.8}
\small\frac{1}{2{\phi}^{3/2}}\hspace{-5pt}\begin{bmatrix}
2{\phi} \Gemini^{\mathbb{R}^3}_\nu (u,u)- \la \nabla_{\mathbb{R}^3} {\phi},\nu\ra_{\mathbb{R}^3} & 0 &  -\la u \times \nabla_{\R^3} {\phi}, \nu\ra_{\mathbb{R}^3}\\
0 & 2{\phi} \Gemini^{\mathbb{R}^3}_{\nu}(v,v)- \la \nabla_{\mathbb{R}^3} {\phi},\nu\ra_{\mathbb{R}^3} & 0\\
- \la u \times \nabla_{\mathbb{R}^3} {\phi}, \nu\ra_{\mathbb{R}^3} &0 & \la \nabla_{\mathbb{R}^3} {\phi},\nu\ra_{\mathbb{R}^3}
\end{bmatrix}.
$$
In particular, it is positive definite, and hence $N_r$ is strictly convex, if and only if we have the following inequalities:
\begin{align} 
&2{\phi} \Gemini^{\mathbb{R}^3}_\nu (u,u)- \la \nabla_{\mathbb{R}^3} {\phi},\nu\ra_{\mathbb{R}^3}>0;\label{E1} \\
&2{\phi} \Gemini^{\mathbb{R}^3}_\nu (v,v)- \la \nabla_{\mathbb{R}^3} {\phi},\nu\ra_{\mathbb{R}^3}>0; \label{E2} \\ 
&(2{\phi} \Gemini^{\mathbb{R}^3}_\nu (u,u)- \la \nabla_{\mathbb{R}^3} {\phi},\nu\ra_{\mathbb{R}^3})\la \nabla_{\mathbb{R}^3} {\phi},\nu\ra_{\mathbb{R}^3}-\la \nabla_{\mathbb{R}^3} {\phi},v\ra^2_{\mathbb{R}^3}>0. \label{E4}
\end{align}

Let's first prove the case $m=0$. Explicitly, it is easy to compute: 
\begin{align*}
(\ref{E1})&=\sum_{\pm} \frac{\cosh^2 r\mp 2\cosh r\cos \beta+1}{2A\sinh r \av{x-p_{\pm}}^2_{\mathbb{R}^3}};\\
(\ref{E2})&=\sum_{\pm} \frac{\sinh r}{2a^2A^3};\\
(\ref{E4})&=\left( \sum_{\pm} \frac{1}{\av{x-p_{\pm}}^2_{\mathbb{R}^3}}\right) \left(\sum_{\pm} \frac{(\cosh r\mp \cos \beta)^2}{4A^2 \av{x-p_{\pm}}^2_{\mathbb{R}^3}} \right),
\end{align*}
which are clearly positive.

Now, we consider the case $m>0$ and we write $\phi=m+\tilde{\phi}$. The first minor of the second fundamental form is positive if and only if $2m \Gemini^{\mathbb{R}^3}_\nu (u,u)+(2\tilde{\phi} \Gemini^{\mathbb{R}^3}_\nu (u,u)- \la \nabla_{\mathbb{R}^3} \phi,\nu\ra_{\mathbb{R}^3})>0$. The first term is clearly greater than zero as  $\Gemini^{\mathbb{R}^3}_\nu (u,u)$ and $m$ are. The positivity of the remaining part follows from (\ref{E1}) in the $m=0$ case and $\nabla\phi=\nabla\tilde{\phi}$. Analogously, using ($\ref{E2}$) with $m=0$, we can prove that the second minor is positive.  

The determinant is greater than zero if and only if 
\begin{align*}
&2m\Gemini^{\mathbb{R}^3}_\nu (u,u)\la \nabla_{\mathbb{R}^3} {\phi},\nu\ra_{\mathbb{R}^3}+\\
 &+\left( (2\tilde{\phi} \Gemini^{\mathbb{R}^3}_\nu (u,u)- \la \nabla_{\mathbb{R}^3} \phi,\nu\ra_{\mathbb{R}^3})\la \nabla_{\mathbb{R}^3} \phi,\nu\ra_{\mathbb{R}^3}-\la \nabla_{\mathbb{R}^3} \phi,v\ra^2_{\mathbb{R}^3}\right)>0.
\end{align*}
This is the case because of ($\ref{E4}$) in the $m=0$ case and $\nabla\phi=\nabla \tilde{\phi}$.

We conclude that $N_r$ is strictly convex for all $r>0$ and all $m\geq0$.
\end{proof}

\begin{remark}
We observe that this proof, in the case $m=0$, is conceptually equivalent to Lemma \ref{EHconvexfunction}. Indeed, as observed in Example \ref{Level set example}, the ellipsoids are the level sets of the square of the distance function from the circle invariant compact minimal surface in the Eguchi--Hanson space. Lemma \ref{EHconvexfunction}, together with Remark \ref{LevelSetsConvexFunctions}, implies that they need to be strictly convex.
\end{remark}

\begin{thm}\label{Theoremnnonexistence2P}
Let $(X,g)$ be a multi-Eguchi--Hanson or a multi-Taub--NUT space with two singular points of $\phi$. Then, compactly supported stationary integral varifolds are contained in the unique $S^1$-invariant compact minimal surface.
\end{thm}
\begin{proof}
The proof follows as in Theorem \ref{SphereBarrierHyp}, where we use Lemma \ref{EllipsoidBarrierLemma} instead of Lemma \ref{SphereBarrierHypLemma}. 
\end{proof}
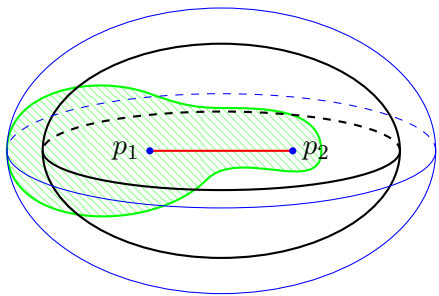
\begin{figure}[t]
\begin{tikzpicture}[scale=0.95]
\draw [-,thick, green, pattern=north west lines, pattern color=green!40] (-3,0) to [out=90,in=160] (-1,0.8)
to [out=-20,in=180] (0,0.6) to [out=0, in=90] (1.4,0)
to [out=-90, in=45] (-0.2,-0.4) to [out=225, in=-90] (-3,0);

\draw[red, thick] (-1,0,0)--(1,0,0);
\fill[blue] (-1,0,0) circle (0.05cm) node[black,left] {$p_1$};
\fill[blue] (1,0,0) circle (0.05cm) node[black,right] {$p_2$};

\draw[black, thick] (0,0,0) ellipse (2.5cm and 1.5cm);
\draw[black, thick, dashed] (2.5,0,0) arc(0:180:2.5cm and 0.55cm);
\draw[black, thick] (-2.5,0,0) arc(180:360:2.5cm and 0.55cm);

\draw[blue] (0,0) ellipse (3cm and 2cm);
\draw[blue, dashed] (3,0,0) arc(0:180:3cm and 0.8cm);
\draw[blue] (-3,0,0) arc(180:360:3cm and 0.8cm);
\end{tikzpicture}
\caption{Barriers used in Theorem \ref{Theoremnnonexistence2P}.}
\end{figure}

The last result is the geometric measure theory generalization of Theorem \ref{IntroTheoremTwoPoints}.

Putting together Theorem \ref{TheoremnonexistenceEH}, Corollary \ref{ClassicalCorollaryR^4Taub-NUT} and Theorem \ref{Theoremnnonexistence2P}, we have:
\begin{corollary}
In the multi-Eguchi--Hanson and multi-Taub--NUT spaces with at most two singular points of $\phi$, compact minimal submanifolds (compactly supported stationary integral varifolds) are $S^1$-invariant, or are contained in one.
\end{corollary}

\begin{remark}
Observe that we are not claiming that all compact minimal submanifolds are circle-invariant. Indeed, as the circle-invariant compact minimal submanifold of the Eguchi--Hanson space (multi-Taub--NUT space with two singular points of $\phi$) is totally geodesic \cite[Lemma 4.2]{LO20}, the closed non-equivariant geodesics of it are also closed geodesics in the total space. By the theorem of the three geodesics there are at least 2 of such objects.
\end{remark}

\begin{remark}
It is easy to see that the results in this subsection are still true for multi-centred Gibbons--Hawking spaces.
\end{remark}

\begin{remark}
In the Euclidean space and in the Taub--NUT space we showed that spheres centred at the origin are strictly convex. Moreover, in the Eguchi--Hanson space and in the two-centred multi-Taub--NUT space, we showed that ellipsoids with foci the singular points of $\phi$ are strictly convex. Since spheres can be considered 1-focus ellipsoids, one would expect k-foci ellipsoids to be strictly convex in the multi-Eguchi--Hanson and multi-Taub--NUT spaces with singular set of $\phi$ corresponding to the foci.

Unfortunately, this cannot hold even in the three point case. Indeed, 3-ellipsoids form a family of (possibly singular when passing through the foci) surfaces that foliates the space and shrinks to a point (see figure \ref{3-ellipsoids}). Clearly, if the surfaces were convex at all non-singular points of $\phi$, we could only have 1 circle-invariant closed geodesic contradicting Proposition \ref{Circle-invariant geodesics}. Moreover, even if the 3-ellipsoids were 2-convex at all non-singular points of $\phi$, this wouldn't be enough to prove that compact minimal surfaces need to be circle invariant.
\end{remark}

\begin{figure}[t]
\centering
\begin{subfigure}{.5\textwidth}
  \centering
  \begin{tikzpicture} [scale=1.3]
\draw[thick, red] (-1,0)--(1,0);
\draw[thick, red] (-1,0)--(0,1.73205);
\draw[thick, red] (1,0)--(0,1.73205);
\draw [-,thick, green] (-1,0) to [out=90, in=205] (0,1.73205);
\draw [-,thick, green, rotate around={120:(0,0.5775)}] (-1,0) to [out=90, in=205] (0,1.73205);
\draw [-,thick, green, rotate around={-120:(0,0.5775)}] (-1,0) to [out=90, in=205] (0,1.73205);
\draw [-,thick, green] (-1.5,0) to [out=110, in=180] (0,2.2) to [out=0, in=70] (1.5,0);
\draw [-,thick, green]  (-1.5,0) to [out=-70, in=-110] (1.5,0);
\draw [-,thick, green, scale=0.2] (-1.5,2.5) to [out=110, in=180] (0,4.7) to [out=0, in=70] (1.5,2.5);
\draw [-,thick, green, scale=0.2]  (-1.5,2.5) to [out=-70, in=-110] (1.5,2.5);
\fill[blue] (-1,0) circle (0.05cm) node[black,left] {$p_1$};
\fill[blue] (1,0) circle (0.05cm) node[black,right] {$p_2$};
\fill[blue] (0,1.73205) circle (0.05cm) node[black,left] {$p_3$};
\end{tikzpicture}
  \end{subfigure}%
\begin{subfigure}{.5\textwidth}
  \centering
  \begin{tikzpicture}[scale=0.7]

\draw[thick, red] (-1,0)--(1,0);
\draw[thick, red] (-1,0)--(0,4);
\draw[thick, red] (1,0)--(0,4);

\draw [-,thick, green] (-1,0) to [out=110, in=180] (0,2) to [out=0, in=70] (1,0);
\draw [-,thick, green, rotate around={120:(0,0.5775)}] (-1,0) to [out=90, in=205] (0,1.73205);

\draw [-,thick, green, scale=0.2] (-1.5,2) to [out=120, in=180] (0,3.7) to [out=0, in=60] (1.5,2);
\draw [-,thick, green, scale=0.2]  (-1.5,2) to [out=-70, in=-110] (1.5,2);

\draw [-,thick, green] (-1.9,-0.4) to [out=120, in=220] (0,4) to [out=-40, in=60] (1.9,-0.4) to [out=-120 ,in=-60] (-1.9,-0.4);

\fill[blue] (-1,0) circle (0.08cm) node[black,left] {$p_1$};
\fill[blue] (1,0) circle (0.08cm) node[black,right] {$p_2$};
\fill[blue] (0,4) circle (0.08cm) node[black,left] {$p_3$};
\end{tikzpicture}
  
  \end{subfigure}
    \caption{Examples of 3-ellipsoids in the plane containing the foci.}
    \label{3-ellipsoids}
\end{figure}
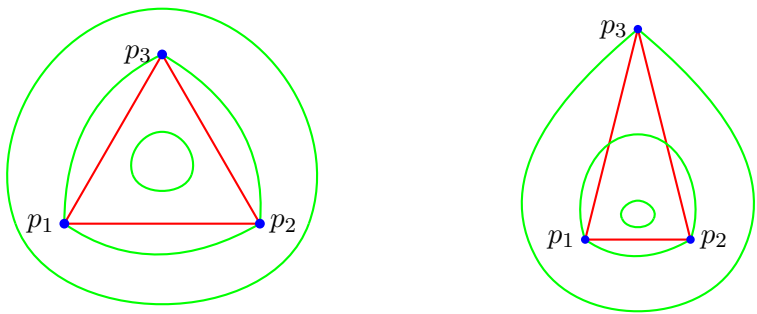

\subsection{Local barriers} In Section \ref{SectionBarrier}, we discussed the connection between strong stability and the convexity of the square of the distance function. We also showed that, in the multi-Eguchi--Hanson and in the multi-Taub--NUT spaces, the only strongly stable compact minimal submanifolds are, essentially, the circle-invariant compact minimal surfaces. In this setting, by Proposition \ref{Strong Stability in HK}, the strong stability condition is completely encoded by the Gaussian curvature of the surface. 

Lotay and Oliveira computed the Gaussian curvature of a circle-invariant compact minimal surface and obtained the following result.

\begin{lemma}[Lotay and Oliveira{\cite[Appendix A]{LO20}}] \label{GaussianCurvature}
Let $(X,g)$ be a multi-Eguchi--Hanson or a multi-Taub--NUT space, let $N$ be a compact $S^1$-invariant minimal surface in $(X,g)$ and let $\gamma:=\pi(N)$ be the associated straight line in $U$ connecting two singular points of $\phi$. Without loss of generality, we can assume that $\gamma$ is the straight line connecting $p_{\pm}:=(0,0,\pm a)$. Then, the Gaussian curvature of $N$ is given by:
\[
K=-\frac{\partial^2}{\partial x_3^2} \left(\frac{1}{2\phi}\right).
\]
Moreover, if we write
$$
\phi=m+\sum_{i=3}^{k} \frac{1}{2\av{x-p_i}_{\mathbb{R}^3}}+\frac{1}{2\av{x-p_+}_{\mathbb{R}^3}}+\frac{1}{2\av{x-p_-}_{\mathbb{R}^3}},
$$
and define
$$
\tilde{\phi}:=m+\sum_{i=3}^{k} \frac{1}{2\av{x-p_i}_{\mathbb{R}^3}},
$$
then, $K$ has the form: 
$$K=-\frac{M+N}{2(a+\tilde{\phi}(a^2-x_3^2))^3},$$ 
where
\[
N:=-(2a^2+2a\tilde{\phi}(a^2-x_3^2)+8a\tilde{\phi}x_3^2)
\]
and where
\begin{align*}
M:&=2(\partial_{x_3} \tilde{\phi})^2(a^2-x_3^2)^3+8a x_3 (\partial_{x_3} \tilde{\phi})(a^2-x_3^2)\\
&\hspace{10pt}-a(\partial^2_{x_3} \tilde{\phi})(a^2-x_3^2)^2-\tilde{\phi}(\partial^2_{x_3} \tilde{\phi})(a^2-x_3^2)^3\\
:&=(I)+(II)+(III)+(IV).
\end{align*}
\end{lemma}
\begin{proof}
This follows from Cartan structure equations and a direct computation. 
\end{proof}

\begin{figure}[t]
\begin{tikzpicture}[scale=0.9]
\fill[black] (-0.6,0) circle (0.05cm) node[black,left] {$p_1$};
\fill[black] (0.6,0) circle (0.05cm) node[black,right] {$p_2$};
\fill[black] (0,0) circle (0.05cm) node[black,below] {$q$};
\draw[thick] (-0.6,0) -- (0.6,0);

\draw[blue,thick] (0,0) circle (2cm);
\draw[thick, blue, dashed] (2,0) arc (0:180:2cm and 0.5 cm);
\draw[blue,thick] (-2,0) arc (180:360:2cm and 0.5 cm);
\draw[red, thick, dashed, rotate=45] (0,0)--(2,0);
\node at (2.2,1.57) [red] {$(s+1)a$};

\fill[black] (-2.3,1.2) circle (0.05cm) node[black,right] {$p_3$};
\fill[black] (0,-2.4) circle (0.05cm) node[black,right] {$p_4$};
\fill[black] (2.3,-0.3) circle (0.05cm) node[black,right] {$p_5$};
\end{tikzpicture}
\caption{Example of distribution of points satisfying the condition given in Proposition \ref{strong stability GH}.}
\end{figure}
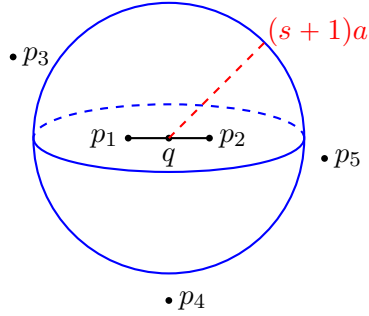

We can use Lemma \ref{GaussianCurvature} to prove Proposition \ref{strong stability GH}.
\begin{proof}[Proof of Proposition \ref{strong stability GH}]
By Proposition \ref{Strong Stability in HK}, it is enough to show that $N$ has positive Gaussian curvature. Moreover, Lemma \ref{GaussianCurvature} implies that it is equivalent to the condition: $$M+N<0.$$ Note that $N$ has always the right sign, so we just need to control the terms of $M$: $(I)$, $(II)$, $(III)$ and $(IV)$. 

Letting $r_l:=\av{x-p_l}_{\mathbb{R}^3}$, we have $\tilde{\phi}=m+\sum_{l=3}^k \frac{1}{2r_l}$,
$\partial_{x_3}\tilde{\phi}=\sum_{l=3}^{k}\frac{(p_l)_3-x_3}{2r_l^3}$ 
and $\partial^2_{x_3} \tilde{\phi}=\sum_{l=3}^{k} \frac{1}{r_l^3}-\frac{3}{2}\sum_{l=3}^{k} \frac{(p_l)_1^2+(p_l)_2^2}{r_l^5}$. Since $$r_l=\sqrt{(p_l)_1^2+(p_l)_2^2+((p_l)_3-x_3)^2}\geq \av{(p_l)_3-x_3}$$ on $\gamma$, we deduce that 
\[
\av{\partial_{x_3}\tilde{\phi}}_{\mathbb{R}^3}\leq \sum_{l=3}^k \frac{1}{2r^2_l}; \hspace{20pt} \partial^2_{x_3} \tilde{\phi}\geq -\sum_{l=3}^k \frac{1}{2r_l^3}.
\]

Defining $b:=\min_{l\geq3} r_l$, it is clear that $b\leq r_l$ for all $l>2$ and hence, $1/r_l\leq 1/b$. Now, we have the obvious estimates:
\begin{align*}
    (I)&\leq 2\left(\frac{k-2}{2b^2}\right)^2 a^6;\\
    (II)&\leq 8\frac{a^2}{b} \left(\sum_{l=3}^k \frac{1}{2r_l}\right)(a^2-x_3^2)\leq8\frac{a^2}{b} \tilde{\phi}(a^2-x_3^2);\\
    (III)&\leq a\left(\sum_{l=3}^k \frac{1}{2r_l^3} \right)a^2 (a^2-x_3^2)\leq \frac{a^3}{b^2}\tilde{\phi}(a^2-x_3^2);\\
    (IV)&\leq \tilde{\phi} \left(\sum_{l=3}^k \frac{1}{2r_l^3}\right) a^4(a^2-x_3^2)\leq \frac{k-2}{2b^3}a^4 \tilde{\phi}(a^2-x_3^2).
\end{align*}
Triangle inequality, together with the conditions on the Euclidean distance from $q$ to $p_i$, gives $b> sa$. Combining it with the previous estimates for $(I), (II), (III), (IV)$, we obtain: 

\begin{align*}
    (I)&< 2\left(\frac{k-2}{2s^2}\right)^2 a^2; \hspace{20pt} (II)<\frac{8}{s} a\tilde{\phi}(a^2-x_3^2);\\
    (III)& < \frac{1}{s^2}a\tilde{\phi}(a^2-x_3^2); \hspace {19pt} (IV)< \frac{k-2}{2s^3}a \tilde{\phi}(a^2-x_3^2).
\end{align*}
Under the assumptions on $s$, it is immediate to see that $(I)-2a^2<0$ and that $(II)+(III)+(IV)-2a\tilde{\phi} (a^2-x_3^2)<0$. We conclude that $M+N<0$.
\end{proof}

\begin{corollary}
Let $(X,g)$ be a multi-Eguchi--Hanson or a multi-Taub--NUT space and let $N$ be a compact $S^1$-invariant minimal surface in $(X,g)$. If $(X,g)$ and $N$ satisfy the conditions of Proposition \ref{strong stability GH}, then, $N$ is the only compact minimal submanifold (compactly supported stationary integral varifold) of dimension at least $2$ in a neighbourhood of $N$.
\end{corollary}
\begin{proof}
The local uniqueness follows from Proposition \ref{TW local uniqueness} and the usual barrier argument.
\end{proof}

\begin{remark}
Since the real root of $-4x^3+16x^2+2x+(k-2)$ is strictly greater than 4, Proposition \ref{strong stability GH} is weaker than \cite[Proposition A.1]{LO20} in the collinear case. 
\end{remark}

\begin{proposition} 
There is no distribution of 3 or more points for which the condition of Proposition \ref{strong stability GH} is satisfied by all compact $S^1$-invariant minimal surfaces.
\end{proposition}
\begin{proof}
It is enough to show that, given 3 points $\{p_1, p_2,p_3\}\subset\mathbb{R}^3$ such that $d_{\mathbb{R}^3} (\frac{p_1+p_2}{2},p_3)>4\frac{\av{p_1-p_2}_{\mathbb{R}^3}}{2}$, then $d_{\mathbb{R}^3} (\frac{p_2+p_3}{2},p_1)<4\frac{\av{p_2-p_3}_{\mathbb{R}^3}}{2}$. This is an easy application of triangle inequality (see Figure \ref{Proof Proposition no distrib}). 
\end{proof}
\begin{remark}
The same holds if we consider \cite[Proposition A.1]{LO20} instead.
\end{remark}

\begin{figure}[t]
\begin{tikzpicture}[scale=0.75]
\draw[thick,red, rotate=-15] (1.5,0) arc (0:81:1.5cm);
\draw[thick,red, rotate=68.5] (1.5,0) arc (0:8.5:1.5cm);
\draw[thick,red, rotate=79.5] (1.5,0) arc (0:10.5:1.5cm);
\draw[thick,red, rotate=93.5] (1.5,0) arc (0:235:1.5cm);
\draw[thick,blue] (0.35,1.15) circle (3cm);

\draw[red, rotate=-45, dashed] (0,0)--(1.5,0);

\node[scale=0.9] at (1.7,-0.51) [red]{$2\lvert p_1 -p_2\rvert$};
\draw[blue, rotate around={45:(0.35,1.15)}, dashed] (0.35,1.15)--(3.35,1.15);
\node[scale=0.9] at (2.25,1.8) [blue]{$2\lvert p_2 -p_3\rvert$};

\draw[thick, green] (-0.5,0) -- (0.5,0);
\draw[thick, green] (0.5,0) -- (0.2,2.3);
\draw[thick, green] (-0.5,0) -- (0.2,2.3);

\fill[black] (-0.5,0) circle (0.05cm) node[black,left] {$p_1$};
\fill[black] (0.5,0) circle (0.05cm) node[black,right] {$p_2$};
\fill[black] (0.2,2.3) circle (0.05cm) node[black,above] {$p_3$};
\fill[black] (0,0) circle (0.05cm) node[black,below] {$q$};
\fill[black] (0.35,1.15) circle (0.05cm) node[black,right] {$r$};
\end{tikzpicture}
\caption{\hspace{5pt}}\label{Proof Proposition no distrib}
\end{figure}

Finally, we use once again Lemma \ref{GaussianCurvature} to prove Proposition \ref{CounterexampleSS}.

\begin{proof}[Proof of Proposition \ref{CounterexampleSS}]
By Lemma \ref{GaussianCurvature}, a direct computation yields: 
\[
(M+N)(p)=-2a^2-2a^3m-\frac{a^3}{\epsilon}+\frac{a^5}{2\epsilon^3}+\frac{ma^6}{2\epsilon^3}+\frac{a^6}{4\epsilon^4},
\]
for all $p\in\pi^{-1}(0)$.

As $\epsilon\xrightarrow[]{}0$, the leading term of $(M+N)(p)$ is $\frac{a^6}{4}>0$. Then, for $\epsilon$ small enough, $(M+N)(p)>0$ and so $K(p)<0$.

Analogously, as $a\xrightarrow[]{}+\infty$, the leading term is $\frac{m}{2\epsilon^3}+ \frac{1}{4\epsilon^4}>0$. Then, for $a$ big enough, $(M+N)(p)>0$ and so $K(p)<0$.
\end{proof}

This result, together with Proposition \ref{Final converse Strong Stability}, implies that any function that locally looks like the distance function, or a function of the distance function, cannot be $2$-convex in this setting. 
Since in all examples where the barrier method is used we only have dependence on the distance function \cite{TW, TW1, TW18}, we have shown that the natural local theory does not work.

%% file: main.bbl
\begin{thebibliography}{Hala12}		

\bibitem[Bie99]{Bie99}{\scshape Bielawski, R.}, Complete hyper-Kähler 4n-manifolds with a local tri-Hamiltonian $\R^n$-action, {\em Math. Ann.} {\bf{314}} (1999), no. 3, 505-528. \mrev{1704547}, \zbl{0952.53024}.


\bibitem[GH78]{GH78}{\scshape Gibbons, G.; Hawking, S.}, Gravitational multi-instantons, {\em Phys. Lett.} {\bf{78B}} (1978), no. 4, 430-432.

\bibitem[GW00]{GW00}{\scshape Gross, M.; Wilson, P. M. H.}, Large complex structure limits of K3 surfaces, {\em J. Differential Geom.} {\bf{55}} (2000), no. 3, 475-546. \mrev{1863732}, \zbl{1027.32021}.

\bibitem[HaLa12]{HaLa12}{\scshape Harvey, F. R.; Lawson, H.B.}, Geometric plurisubharmonicity and convexity: an introduction, {\em Adv. Math.}, {\bf{230}} (2012), no. 4-6, 2428-2456. \mrev{2927376}, \zbl{1251.31003}.

\bibitem[HsLa71]{HsLa71}{\scshape Hsiang, W.; Lawson, H.B.}, Minimal submanifolds of low cohomogeneity, {\em J. Differential Geom.} {\bf{5}} (1971), no. 1-2, 1-38. \mrev{298593}, \zbl{0219.53045}.

\bibitem[KM05]{KM05}{\scshape Karigiannis, S.; Min-Oo, M.}, Calibrated subbundles in noncompact manifolds of special holonomy, {\em Ann. Global Anal. Geom.} {\bf{28}} (2005), 371-394. \mrev{2199999}, \zbl{1093.53054}.

\bibitem[LO20]{LO20}{\scshape Lotay, J. D.; Oliveira, G.}, Special Lagrangians, Lagrangian mean curvature flow and the Gibbons-Hawking ansatz, {\em J. Differential Geom.}, to appear.

\bibitem[LoSu20]{LS20}{\scshape Lotay, J. D.; Schulze, F.}, Consequences of strong stability of minimal submanifolds, {\em Int. Math. Res. Not.} (2020), 2352-2360. \mrev{4090742}, \zbl{1437.53072}.

\bibitem[MW93]{MW93}{\scshape Micallef, M. J.; Wolfson, J. G.}, The second variation of area of minimal surfaces in four-manifolds,  {\em Math. Ann.} {\bf{295}} (1993), 245-268.\mrev{1202392}, \zbl{0788.58016}.

\bibitem[Oli]{Oli}{\scshape Oliveira, G.}, Locating closed geodesics on Foscolo's K3 surfaces, {\em preprint}.

\bibitem[Pr79]{Pr79}{\scshape Prasad, M. K.}, Equivalence of Eguchi-Hanson metric to two center Gibbons-Hawking metric, {\em Phys. Lett.} {\bf{83B}} (1979), no. 3-4, 310.

\bibitem[Sim68]{Sim68}{\scshape Simons, J.}, Minimal varieties in Riemannian manifolds, {\em Annals of Math.} {\bf{88}} (1968), no. 1, pp. 62-105. \mrev{233295}, \zbl{0181.49702}.

\bibitem[TW18]{TW18}{\scshape Tsai, C.-J.; Wang, M.-T.}, Mean curvature flow in manifolds of special holonomy, {\em J. Differential Geom.} {\bf{108}} (2018), no. 3, 531-569. \mrev{3770850}, \zbl{1385.53061}.

\bibitem[TW]{TW}{\scshape Tsai, C.-J.; Wang, M.-T.}, A strong stability condition on minimal submanifolds and its implications, {\em J. Reine Angew. Math.}, {\bf 764} (2020),111-156. \mrev{4116634}, \zbl{07225383}.

\bibitem[TW1]{TW1}{\scshape Tsai, C.-J.; Wang, M.-T.}, Global uniqueness of the minimal sphere in the Atiyah--Hitchin manifold, {\em Preprint: \url{https://arxiv.org/abs/2002.10391}} (2018).	


\bibitem[Wh15]{Wh15}{\scshape White, B.}, Topics in mean curvature flow, notes by O. Chodosh, {\em available at \url{http://web.stanford.edu/~ochodosh/MCFnotes.pdf}} (2015).	

	
            	 \end{thebibliography}
